\documentclass[twoside,11pt]{article}

\usepackage{jmlr2e}

\usepackage{graphicx} % more modern
\usepackage{subfigure}

\usepackage{algorithm}
\usepackage{algorithmic}
\RequirePackage{amsmath}
\RequirePackage{amssymb}

\usepackage[colorlinks=false,allbordercolors={1 1 1}]{hyperref}

\newenvironment{itemizeReduced}{
\begin{list}{\labelitemi}{\leftmargin=2.5em}
\setlength{\itemsep}{3pt}
\setlength{\parskip}{0pt}
\setlength{\parsep}{0pt}}{\end{list}
}

\newcommand{\BEAS}{\begin{eqnarray*}}
\newcommand{\EEAS}{\end{eqnarray*}}
\newcommand{\BEA}{\begin{eqnarray}}
\newcommand{\EEA}{\end{eqnarray}}
\newcommand{\BEQ}{\begin{equation}}
\newcommand{\EEQ}{\end{equation}}
\newcommand{\BIT}{\begin{itemizeReduced}}
\newcommand{\EIT}{\end{itemizeReduced}}
\newcommand{\BNUM}{\begin{enumerate}}
\newcommand{\ENUM}{\end{enumerate}}
\newcommand{\BA}{\begin{array}}
\newcommand{\EA}{\end{array}}

\newcommand{\tr}{\mathop{ \rm tr}}

\newcommand{\rb}{\mathbb{R}}

\newcommand{\mysec}[1]{Section~\ref{sec:#1}}
\newcommand{\myapp}[1]{Appendix~\ref{app:#1}}
\newcommand{\eq}[1]{Equation~(\ref{eq:#1})}

\def \E{{\mathbb E}}
\def \P{{\mathbb P}}

\def \E{{\mathbb E}}
\def \P{{\mathbb P}}
\def \F{{\mathcal F}}
\def \H{{\mathcal H}}

\jmlrheading{15}{2014}{367-399}{10/13; Revised 12/13}{2/14}{Francis Bach}

\ShortHeadings{Adaptivity of Averaged Stochastic Gradient Descent}{Bach}
\firstpageno{367}

\begin{document}

\title{Adaptivity of Averaged Stochastic Gradient Descent \\to Local Strong Convexity for Logistic Regression}

\author{\name Francis Bach \email francis.bach@ens.fr \\
       \addr INRIA - Sierra Project-team\\
       D\'epartement d'Informatique de l'Ecole Normale Sup\'erieure \\
     Paris, France
     }

\editor{L\'eon Bottou}

\maketitle

\begin{abstract}%   <- trailing '%' for backward compatibility of .sty file
 In this paper, we consider supervised learning problems such as logistic regression and study the stochastic gradient method with averaging, in the usual stochastic approximation setting where observations are used only once. We show that after $N$ iterations,   with a constant step-size proportional to $1/R^2 \sqrt{N}$ where $N$ is the number of observations and $R$ is the maximum norm of the observations,   the convergence rate is always of order $O(1/\sqrt{N})$, and improves to $O(R^2 / \mu N)$ where $\mu$ is the lowest eigenvalue of the Hessian at the global optimum (when this eigenvalue is greater than $R^2/\sqrt{N}$). Since $\mu$ does not need to be known in advance, this shows that averaged stochastic gradient is adaptive to \emph{unknown  local} strong convexity of the objective function. Our proof relies on the generalized self-concordance properties of the logistic loss and thus extends to all generalized linear models with uniformly bounded features.

 \end{abstract}


\section{Introduction}

The minimization of an objective function which is only available through unbiased estimates of the function values or its gradients is a key methodological  problem in many disciplines. Its analysis has been attacked mainly in three scientific communities: stochastic approximation \citep{fabian1968asymptotic,ruppert,polyak1992acceleration,kushner:yin:2003,broadie2009general},
optimization \citep{nesterov2008confidence,nemirovski2009robust}, and  machine learning \citep{bottou2005line,shalev2007pegasos,bottou-bousquet-2008b,shalev2008svm,shalev2009stochastic,duchi,xiao2010dual}. The main algorithms which have emerged are stochastic gradient descent (a.k.a.~Robbins-Monro algorithm), as well as a simple modification where iterates are averaged (a.k.a.~Polyak-Ruppert averaging). 

For convex optimization problems, the convergence rates of these algorithms depends primarily on the potential \emph{strong convexity} of the objective function \citep{nemirovsky1983problem}. For $\mu$-strongly convex functions, after $n$ iterations (i.e., $n$ observations), the optimal rate of convergence of function values is $O(1/  \mu n)$ while for convex functions the optimal rate is $O(1/\sqrt{n})$, both of them achieved by averaged stochastic gradient with step size respectively proportional to $1/ \mu n$ or $1/\sqrt{n}$ \citep{nemirovsky1983problem,agarwal2010information}. For smooth functions, averaged stochastic gradient with step sizes proportional to $1/\sqrt{n}$ achieves them up to logarithmic terms \citep{gradsto}.

Convex optimization problems coming from supervised machine learning are typically of the form $f(\theta) = \E \big[ \ell ( y , \langle \theta, x \rangle) \big]$,
where $\ell ( y , \langle \theta, x \rangle)$ is the loss between the response $y \in \rb $ and the prediction $\langle \theta, x \rangle \in \rb$, where $x$ is the input data   in a Hilbert space $\H$ and linear predictions  parameterized by $\theta \in \H$ are considered. They may or may not have strongly convex objective functions. This most often depends on (a) the correlations between covariates $x$, and (b) the strong convexity of the loss function $\ell$. The logistic loss  $\ell: u \mapsto \log (1 + e^{-u})$ is not strongly convex unless restricted to a compact set (indeed, restricted to $u \in [-U,U]$, we have
$\ell''(u) = e^{-u} ( 1 + e^{-u})^{-2} \geqslant \frac{1}{4}e^{-U} $).
Moreover, in the sequential observation model, the correlations are not known at training time. Therefore, many theoretical results based on strong convexity do not apply (adding a squared norm $\frac{\mu}{2} \| \theta\|^2$ is a possibility, however, in order to avoid adding too much bias, $\mu$ has to be small and typically much smaller than $ {1}/{\sqrt{n}}$, which then makes all strongly-convex bounds vacuous). The goal of this paper is to show that with   proper assumptions, namely self-concordance, one can readily obtain favorable theoretical guarantees for logistic regression, namely a rate of the form $O(R^2 / \mu n)$ where $\mu$ is the lowest eigenvalue of the Hessian at the global optimum, \emph{without any exponentially increasing constant factor} (e.g.,  with the notations above, without terms of the form $e^{U}$).

Another goal of this paper is to design an algorithm and provide an analysis that benefit from \emph{hidden} local strong convexity without requiring to know the local strong convexity constant in advance. In smooth situations, the results of \citet{gradsto} imply that the averaged stochastic gradient method with step sizes of the form $O(1/\sqrt{n})$ is adaptive to the strong convexity of the problem. However the dependence in $\mu$ in the strongly convex case is of the form $O( 1/\mu^{2} n)$, which is sub-optimal. Moreover, the final rate is rather complicated, notably because all possible step-sizes are considered. Finally, it does not apply here because even in low-correlation settings, the objective function of logistic regression cannot be globally strongly convex.

In this paper, we provide an analysis for stochastic gradient with averaging for generalized linear models such as logistic regression, with a step size proportional to $1/ R^2 \sqrt{n}$ where $R$ is the radius of the data and $n$ the number of observations, showing such adaptivity. In particular, we show that the algorithm can adapt to the \emph{local} strong-convexity constant, that is, the lowest eigenvalue of the Hessian at the optimum. The analysis is done for a finite horizon $N$ and a constant step size decreasing in~$N$ as $1/R^2 \sqrt{N}$, since the analysis is then slightly easier, though (a) a decaying stepsize could be considered as well, and (b) it could be classically extended to varying step-sizes by a doubling trick \citep{hazanbeyond}.

\section{Stochastic Approximation for Generalized Linear Models}
In this section, we present the assumptions our work relies on, as well as related work.

\subsection{Assumptions}
Throughout this paper, we make the following assumptions.
We consider  a function $f$ defined on a Hilbert space $\H$, equipped with a norm $\| \cdot \|$. Throughout the paper, we identify the Hilbert space and its dual; thus, the gradients of $f$ also belongs to $\H$ and we use the same norm on these. Moreover, we consider
an increasing family of $\sigma$-fields $(\mathcal{F}_n)_{n \geqslant 1}$ and we assume that we are given a deterministic
$\theta_0 \in \H$, and a sequence of functions $f_n: \H \to \rb$, for $n\geqslant 1$. We  make the following assumptions, for a certain $R>0$:
\BIT
\item[\textbf{(A1)}] \textbf{Convexity and differentiability of $f$}: $f$ is convex and three-times differentiable.
\item[\textbf{(A2)}] \textbf{Generalized self-concordance of $f$} \citep{bach2010self}:  for all $\theta_1,\theta_2 \in \H$, the function $\varphi: t \mapsto f\big[\theta_1 + t (\theta_2 - \theta_1) \big]$ satisfies: $\forall t \in \rb$, $| \varphi'''(t) |\leqslant R \| \theta_1 - \theta_2 \| \varphi''(t)$.
\item[\textbf{(A3)}] \textbf{Attained global minimum}: $f$ has a global minimum attained at $\theta_\ast \in \H$.
\item[\textbf{(A4)}] \textbf{Lipschitz-continuity of $f_n$ and $f$}:  all gradients of $f$ and $f_n$ are bounded by $R$, that is, 
for all $\theta\in \H$,
$$\| f'(\theta)  \|
\leqslant R \mbox{ and } \forall n \geqslant 1, \ \| f_n'(\theta) \|
\leqslant R \mbox{ almost surely}.$$

\item[\textbf{(A5)}] \textbf{Adapted measurability}: $\forall n \geqslant 1$, $f_n$ is $\mathcal{F}_n$-measurable.
\item[\textbf{(A6)}] \textbf{Unbiased gradients}: $\forall n \geqslant 1$, $\E( f_n'(\theta_{n-1})  | \F_{n-1} ) = f'(\theta_{n-1})$.
\item[\textbf{(A7)}] \textbf{Stochastic gradient recursion}: $\forall n \geqslant 1$, $\theta_n = \theta_{n-1} - \gamma_n f_n'(\theta_{n-1})$, where $(\gamma_n)_{n \geqslant 1}$ is a deterministic sequence.\EIT
In this paper, we will also consider the averaged iterate $\bar{\theta}_n = \frac{1}{n} \sum_{k=0}^{n-1} \theta_k$, which may be trivially computed on-line through the recursion $\bar{\theta}_n = \frac{1}{n} \theta_{n-1} + \frac{n-1}{n} \bar{\theta}_{n-1}$.

Among the seven assumptions above, the non-standard one is \textbf{(A2)}: the notion of self-concordance is an important tool in convex optimization and in particular for the study of Newton's method  \citep{self}. It corresponds to having the third derivative bounded by the $\frac{3}{2}$-th power of the second derivative. For machine learning, \citet{bach2010self} has generalized  the notion of self-concordance by removing the  $\frac{3}{2}$-th power, so that it is applicable to cost functions arising from probabilistic modeling, as shown below. The key consequence of our notion of self-concordance is a relationship shown in Lemma~\ref{prop:selfc} (\mysec{strong}) between the norm of a gradient $\| f'(\theta)\|$ and the excess cost function $f(\theta) -f(\theta_\ast)$, which is the same than for strongly convex functions, but with the local strong convexity constant rather than the global one (which is equal to zero here).

         Our set of assumptions corresponds to the following examples (with i.i.d.~data, and $\mathcal{F}_n$ equal to the $\sigma$-field generated by $x_1,y_1,\dots,x_n,y_n$):
      \BIT
           \item[--] \textbf{Logistic regression}: $f_n(\theta) =  \log( 1+ \exp(- y_n \langle x_n , \theta \rangle  ))$,  with data $x_n$  uniformly almost surely bounded by $R$ and $y_n \in \{-1,1\}$.
           The norm considered here is also the norm of the Hilbert space.
             Note that this includes other binary classification losses, such as $f_n(\theta) =   - y_n \langle x_n , \theta \rangle  + \sqrt{ 1
           +  \langle x_n , \theta \rangle^2}$.

      \item[--] \textbf{Generalized linear models with  uniformly bounded features}: $f_n(\theta) =$ \linebreak[4] $- \langle \theta, \Phi(x_n,y_n) \rangle
      + \log \int h(y) \exp \big( \langle \theta, \Phi(x_n,y) \rangle \big) dy$, with  $\Phi(x_n,y) \in \H$ almost surely bounded in norm by $R$, for all observations $x_n$ and all potential responses $y$ in a measurable space. This includes multinomial regression and conditional random fields \citep{lafferty2001conditional}.  
\item[--] \textbf{Robust regression}: we may use $f_n(\theta) = \varphi( y_n - \langle x_n , \theta \rangle  )$, with $\varphi(t) = \log \cosh t = \log \frac{e^{t} + e^{-t} }{2}$, with a similar boundedness assumption on $x_n$.
      
      \EIT

      \subsection{Running-time Complexity} 
      
      The stochastic gradient descent recursion $\theta_n = \theta_{n-1} - \gamma_n f_n'(\theta_{n-1})$ operates in full generality in the potentially infinite-dimensional Hilbert space $\H$. There are two practical set-ups where this recursion can be implemented. When $\H$ is finite-dimensional with dimension $d$, then the complexity of a single iteration is $O(d)$, and thus $O(dn)$ after $n$ iterations. When $\H$ is infinite-dimensional,  the recursion can be readily implemented when (a) all functions $f_n$ depend on one-dimensional projections $\langle x_n, \theta \rangle$, that is, are of the form $f_n(\theta) = \varphi_n \big( \langle x_n, \theta \rangle \big)$ for certain random functions $\varphi_n$ (e.g., $\varphi_n(u) = \ell(y_n,u)$ in machine learning), and (b)
       all scalar products $K_{ij} = \langle x_i, x_j \rangle$ between $x_i$ and $x_j$, for $i, j \geqslant 1$, can be computed. This may be done through the classical application of the ``kernel trick'' \citep{scholkopf-smola-book,Cristianini2004}: if $\theta_0=0$, we may represent $\theta_n$ as a linear combination of vectors $x_1,\dots,x_n$, that is, $\theta_n = \sum_{i=1}^n 
      \alpha_{i} x_i$, and the recursion may be written in terms of the weights $\alpha_n$, through
    $$\alpha_{n} = - \gamma_n x_n \varphi_n' \bigg( \sum_{i=1}^{n-1} \alpha_i K_{ni} \bigg) .$$
      A key element to notice here is that without regularization, the weights $\alpha_i$ corresponding to previous observations remain constant. The overall complexity of the algorithm is $O(n^2)$ times the cost of evaluating a single kernel function.
See \cite{bordes2005fast} and \citet{wang2012breaking} for approaches aiming at reducing the computational load in this setting. Finally, note that in the kernel setting, the function $f(\theta)$ cannot be strongly convex because the covariance operator of $x$ is typically a compact operator, with a sequence of eigenvalues tending to zero (some regularization is then needed).

 \section{Related Work}
       
       In this section, we review related work, first for non-strongly convex problems then for strongly convex problems.
  
\subsection{Non-strongly-convex Functions} 
When only convexity of the objective function is assumed,  several authors \citep{nesterov2008confidence,nemirovski2009robust,shalev2009stochastic,xiao2010dual} have shown that using a step-size proportional to $1/\sqrt{n}$, \emph{together with some form of averaging}, leads to the minimax optimal rate of $O(1/\sqrt{n})$ \citep{nemirovsky1983problem,agarwal2010information}. Without averaging, the known convergences rates are suboptimal, that is, averaging is key to obtaining the optimal rate \citep{gradsto}. Note that the smoothness of the loss does not change the rate, but may help to obtain better constants, with the potential use of acceleration \citep{lan2010optimal}. Recent work \citep{newsto} has considered algorithms which improve on the rate $O(1/\sqrt{n})$ for smooth self-concordant losses, such as the square and logistic losses. Their analysis relies on some of the results proved in this paper (in particular the high-order bounds in \mysec{nonstronglycvx}).

The compactness of the domain is often used within the algorithm (by using orthogonal projections) and within the analysis (in particular to optimize the step size and obtain high-probability bounds). In this paper, we do not make such compactness assumptions, since in a machine learning context, the available bound would be loose and hurt  practical performance. Note that the analysis of the related dual averaging methods \citep{nesterov2009primal,xiao2010dual} has also been carried without compactness assumptions, and previous analyses would also go through in the same set-up for stochastic mirror descent \citep{nemirovsky1983problem}, at least for bounds in expectation. In the present paper, we  derive higher-order bounds and bounds in high-probability where the lack of compactness is harder to deal with.

Another difference between several analyses is the use of decaying step sizes of the form $\gamma_n \propto 1/\sqrt{n}$ vs.~the use of a constant step size of the form $\gamma \propto 1 / \sqrt{N}$ for a finite known horizon $N$ of iterations. The use of a ``doubling trick'' as done by \citet{hazanbeyond} for strongly convex optimization, where a constant step size is used for iterations between $2^p$ and $2^{p+1}$, with a constant that is proportional to $1/\sqrt{2^p}$, would allow to obtain an anytime algorithm from a finite horizon one. In order to simplify our analysis, we only consider a finite horizon $N$ and a constant step-size that will be proportional to $1/\sqrt{N}$.  

\subsection{Strongly-convex Functions} 

When the function is $\mu$-strongly convex, that is, $\theta \mapsto f(\theta) - \frac{\mu}{2} \| \theta\|^2$ is convex, there are essentially two approaches to obtaining the minimax-optimal rate of  $O(1/ \mu n)$ \citep{nemirovsky1983problem,agarwal2010information}: (a) using a step size proportional to $1/\mu n$ with averaging for non-smooth problems \citep{nesterov2008confidence,nemirovski2009robust,xiao2010dual,shalev2009stochastic,duchi,lacoste2012simpler} or a step size proportional to $1/(R^2 + n \mu)$  also with averaging, for smooth problems, where $R^2$ is the smoothness constant of the loss of a single observation \citep{sag}; (b) for smooth problems, using longer step-sizes proportional to $1/n^\alpha$ for $\alpha \in (1/2,1)$ \emph{with} averaging \citep{polyak1992acceleration,ruppert,gradsto}.

Note that the often advocated step size, that is, of the form $C/n$ where $C$ is larger than $1/\mu$, leads, without averaging to a convergence rate of $O(1/\mu^2 n)$ \citep{fabian1968asymptotic,gradsto}, hence with  a worse dependence on $\mu$.

 The   solution (a) requires to have a good estimate of the strong-convexity constant $\mu$, while the second solution (b) does not require to know such estimate and leads to a convergence rate achieving asymptotically the Cramer-Rao lower bound \citep{polyak1992acceleration}. Thus, this last solution is adaptive to unknown (but positive) amount of strong convexity. However, unless we take the limiting setting $\alpha=1/2$, it is not adaptive to lack of strong convexity. While the non-asymptotic analysis of \citet{gradsto} already gives a convergence rate in that situation, the bound is rather complicated and also has a suboptimal dependence on $\mu$. Another goal of this paper is to consider a less general result, but   more compact and, as already mentioned, a better dependence on the strong convexity constant $\mu$ (moreover, as reviewed below, we consider the \emph{local} strong convexity constant, which is much larger).
 
Finally, note that unless we restrict the support, the objective function for logistic regression cannot be globally strongly convex (since the Hessian tends to zero when $\|\theta\|$ tends to infinity). In this paper we show that stochastic gradient descent with averaging is adaptive to the \emph{local} strong convexity constant, that is, the lowest eigenvalue of the Hessian of $f$ at the global optimum, without any exponential terms in $R D $ (which would be present if a compact domain of diameter $D$ was imposed and traditional analyses were performed).

 \subsection{Adaptivity to Unknown Constants} 
 
 The desirable property of adaptivity to the difficulty of an optimization problem has also been studied in several settings. Gradient descent with constant step size is for example naturally adaptive to the strong convexity of the problem \citep[see, e.g.,][]{nesterov2004introductory}. In the stochastic context, \citet{juditsky2010primal} provide another strategy than averaging with longer step sizes, but for uniform convexity constants.

\section{Non-Strongly Convex Analysis}
\label{sec:nonstronglycvx}
In this section, we study the averaged stochastic gradient method in the non-strongly convex case, that is, without any (global or local) strong convexity assumptions. We first recall existing results in \mysec{existing}, that bound the expectation of the excess risk leading to a bound in $O(1/\sqrt{N})$. We then show using martingale moment inequalities how all higher-order moments may be bounded in \mysec{high}, still with a rate of $O(1/\sqrt{N})$. However, in \mysec{grad}, we consider the convergence of the squared gradient, with now a rate of $O(1/N)$. This last result is key to obtaining the adaptivity to local strong convexity in \mysec{self}.

\subsection{Existing Results}
\label{sec:existing}
In this section, we review existing results for Lipschitz-continuous non-strongly convex problems \citep{nesterov2008confidence,nemirovski2009robust,shalev2009stochastic,duchi,xiao2010dual}. Note that smoothness is not needed here. We consider a constant step size $\gamma_n = \gamma > 0$, for all $n \geqslant 1$, and we denote by $\bar{\theta}_n = \frac{1}{n} \sum_{k=0}^{n-1} \theta_k$ the averaged iterate.

We prove the following proposition, which provides a bound on the expectation of $f(\bar{\theta}_{n}) - f(\theta_\ast)$ that decays at rate $O(\gamma + 1/ \gamma n)$, hence the usual choice $ \gamma \propto 1/\sqrt{n}$:
\begin{lemma}
\label{prop:old}
Assume \textbf{(A1)} and \textbf{(A3-7)}.
With constant step size equal to $\gamma$, for any $n \geqslant 0$, we have:
$$
\E f\bigg( \frac{1}{n} \sum_{k=1}^n \theta_{k-1} \bigg) - f(\theta_\ast) 
 + \frac{1}{2\gamma n} \E \| \theta_{n }- \theta_\ast \|^2
\leqslant   \frac{1}{2\gamma n} \| \theta_{0 }- \theta_\ast \|^2
+ \frac{\gamma}{2}R^2.
$$
\end{lemma}
\begin{proof}
      We have the following recursion, obtained from the Lipschitz-continuity of $f_n$:
      \BEAS
      \| \theta_n  - \theta_\ast \|^2 & =  & \| \theta_{n-1}  - \theta_\ast \|^2  - 2 \gamma  \langle \theta_{n-1} - \theta_\ast ,  f_n'(\theta_{n-1}) \rangle + \gamma^2  \|  f_n'(\theta_{n-1})\|^2 \\
  & \leqslant &  \| \theta_{n-1}  - \theta_\ast \|^2  - 2 \gamma    \langle \theta_{n-1} - \theta_\ast ,  f'(\theta_{n-1}) \rangle +   \gamma^2 R^2 + M_n,
     \EEAS
      with 
      $$
      M_n =  -  2 \gamma  \langle
 \theta_{n-1} - \theta_\ast ,  f_n'(\theta_{n-1}) -  f'(\theta_{n-1}) \rangle.$$

 We thus get, using the classical result from convexity $f(\theta_{n-1}) - f(\theta_\ast)  \leqslant \langle \theta_{n-1} - \theta_\ast ,  f'(\theta_{n-1}) \rangle$:
 \BEQ
 \label{eq:mn}
 2\gamma  \big[ f(\theta_{n-1}) - f(\theta_\ast) \big]
   \leqslant   \| \theta_{n-1}  - \theta_\ast \|^2 - \| \theta_n  - \theta_\ast \|^2   +   \gamma^2 R^2 + M_n.
 \EEQ
Summing over integers less than $n$, this implies:
\BEAS
\frac{1}{n}\sum_{k=0}^{n-1} f(\theta_k) - f(\theta_\ast) + \frac{1}{2\gamma n} \| \theta_{n }- \theta_\ast \|^2
& \leqslant & \frac{1}{2\gamma n} \| \theta_{0 }- \theta_\ast \|^2
+ \frac{\gamma}{2}R^2 + \frac{1}{2 \gamma n} \sum_{k=1}^n M_k.
\EEAS
We get the desired result by taking expectation in the last inequality, and using
the expectation $\E M_k = \E ( \E( M_k | \F_{k-1})) = 0$ and $f\big( \frac{1}{n} \sum_{k=0}^{n-1} \theta_k \big) \leqslant \frac{1}{n} \sum_{k=0}^{n-1} f(\theta_k)$.
\end{proof}
The following corollary considers a specific choice of the step size (note that the bound is only true for the last iterate):
\begin{corollary}
Assume \textbf{(A1)} and \textbf{(A3-7)}.
With constant step size equal to $\gamma = \frac{1}{2R^2\sqrt{N}}$, we have:
\BEAS
\forall n \in \{1,\dots,N\},  \ \  \E \| \theta_{n }- \theta_\ast \|^2
& \leqslant  &  \| \theta_{0 }- \theta_\ast \|^2
+ \frac{1}{4R^2} , \\
  \E f\bigg( \frac{1}{N} \sum_{k=1}^N \theta_{k-1} \bigg) - f(\theta_\ast) 
& \leqslant &  \frac{R^2}{\sqrt{N} } \| \theta_{0 }- \theta_\ast \|^2
+ \frac{1}{4 \sqrt{N}} .
\EEAS
\end{corollary}
  Note that if $\| \theta_0 - \theta_\ast\|^2$ was known, then a better step-size would be $\gamma = \frac{ \| \theta_0 - \theta_\ast\|}{R \sqrt{N}}$, leading to a convergence rate proportional to $\frac{R \| \theta_0 - \theta_\ast\| }{\sqrt{N}}$. However, this requires an estimate (or simply an upper-bound) of $\| \theta_0 - \theta_\ast\|^2$, which is typically not available.

 We are going to improve this result in several ways:
 \BIT
 \item[--] All moments of $\| \theta_{n }- \theta_\ast \|^2$ and $ f (  \bar{\theta}_n ) - f(\theta_\ast) $ will be bounded, leading to a sub-exponential behavior. Note that we do not assume that the iterates are restricted to a predefined bounded set, which is the usual assumption made to derive tail bounds for stochastic approximation \citep{nesterov2008confidence,nemirovski2009robust,kakade2009generalization}.
 \item[--] We are going to show that the squared norm of the gradient at 
 $\bar{\theta}_n = \frac{1}{n} \sum_{k=1}^n \theta_{k-1}$ converges at rate $O(1/n)$, even in the non-strongly convex case. This will  allow us to derive finer convergence rates in presence of local strong convexity in \mysec{strong}.
 \item[--] The bounds above do not explicitly depend on the dimension of the problem, however, in practice, the quantity $R^2 \| \theta_0 - \theta_\ast\|^2$ typically \emph{implicitly} scales linearly in the problem dimension.
 \EIT

\subsection{Higher-Order and Tail Bound}
\label{sec:high}
In this section, we prove novel higher-order bounds (see the proof in Appendix~\ref{app:boundp}), both for any constant step-sizes and then for the specific choice   $\gamma = \frac{1}{2 R^2 \sqrt{N}}$. This will immediately lead to tail bounds.

\begin{proposition}
\label{prop:boundp}
Assume \textbf{(A1)} and \textbf{(A3-7)}.
With constant step size equal to $\gamma$, for any $n \geqslant 0$ and integer $p  \geqslant 1 $, we have:
$$ \E \bigg(
2\gamma n  \big[  f ( \bar{\theta}_n) -   f(\theta^\ast)  \big]+\| \theta_n - \theta_\ast \|^{2} \bigg)^p
 \leqslant  \big( 3 \|\theta_0 - \theta_\ast \|^2 +  20 n p \gamma^2 R^2  \big)^{p}.
$$
\end{proposition}
\begin{corollary}
Assume \textbf{(A1)} and \textbf{(A3-7)}.
With constant step size equal to $\gamma = \frac{1}{2R^2\sqrt{N}}$, for any integer $p  \geqslant 1$, we have:
\BEAS
\forall n \in \{1,\dots,N\}, \ \ 
   \E \| \theta_{n }- \theta_\ast \|^{2p}
& \leqslant  & \Big[ \frac{1}{R^2} \big( 3 R^2\| \theta_{0 }- \theta_\ast \|^2
+ 5 p \big) \Big]^p ,\\
 \E \big[
  f ( \bar{\theta}_N) -   f(\theta^\ast)   \big]^p
& \leqslant & 
\Big[ \frac{1}{\sqrt{N} } \big(  3  R^2\| \theta_{0 }- \theta_\ast \|^2
+ 5 p   \big) \Big]^p.
\EEAS
\end{corollary}
In Appendix~\ref{app:boundp}, we first provide two alternative proofs of the same result: (a) our original somewhat tedious proof based on taking powers of the inequality in \eq{mn} and using martingale moment inequalities, (b) a shorter proof later derived by \citet{newsto}, that uses Burkholder-Rosenthal-Pinelis  inequality \citep[Theorem 4.1]{pinelis}. We also  provide in  Appendix~\ref{app:boundp} a direct proof of the large deviation bound that we now present.

Having a bound on all moments allows immediately to derive large deviation bounds in the same two cases (by applying Lemma~\ref{lemma:conc} from Appendix~\ref{app:proba}):
 \begin{proposition}
 \label{prop:dev}
 Assume \textbf{(A1)} and \textbf{(A3-7)}.
With constant step size equal to $\gamma$, for any $n \geqslant 0$ and $ t \geqslant 0$, we have:
 \BEAS
 \P\Big(
  f ( \bar{\theta}_n) -   f(\theta_\ast)  \geqslant  30 \gamma R^2 t + \frac{3 \| \theta_0 - \theta_\ast\|^2}{\gamma n} \Big)
  \leqslant 2 \exp (-t) , \\
  \P\Big(
  \|  {\theta}_n  -   \theta_\ast \|^2  \geqslant   60 n \gamma^2 R^2  t +  6 \| \theta_0 - \theta_\ast\|^2
 \Big)
  \leqslant 2 \exp (-t) . \EEAS 
 \end{proposition}
 \begin{corollary}
 Assume \textbf{(A1)} and \textbf{(A3-7)}.
With constant step size equal to $\gamma = \frac{1}{2R^2\sqrt{N}}$, for any  $ t \geqslant 0$ we have:
 \BEAS
 \P\Big(
  f ( \bar{\theta}_N) -   f(\theta_\ast)  \geqslant  \frac{15 t }{ \sqrt{N}}  + \frac{ 6R^2 \| \theta_0 - \theta_\ast\|^2}{\sqrt{N}} \Big)
  \leqslant 2 \exp (-t) , \\
  \P\Big(
  \|  {\theta}_N  -   \theta_\ast \|^2  \geqslant    {15}{}   R^{-2}  t +  6 \| \theta_0 - \theta_\ast\|^2
 \Big)
  \leqslant 2 \exp (-t).
   \EEAS 
 \end{corollary}

 We can make the following observations:
\BIT
\item[--] The results above are obtained by direct application of Proposition~\ref{prop:boundp}. In \myapp{boundp}, we also provide an alternative direct proof of a slightly weaker result, which was suggested and  outlined by Alekh Agarwal (personal communication), and that uses 
Freedman's inequality for martingales \citep[Theorem 1.6]{freedman}.

\item[--] The results above bounding the  norm between the last iterate and a global optimum extend to the averaged iterate.

\item[--] The iterates $\theta_n$ and $\bar{\theta}_n$ do not necessarily converge to $\theta_\ast$ (note that $\theta_\ast$ may not be unique in general anyway).
\item[--] Given that $(\E [ f(\bar{\theta}_n) - f(\theta_\ast) ]^p)^{1/p}$ is affine in $p$, we obtain a subexponential behavior, that is, tail bounds similar to an exponential distribution. The same decay was obtained by \citet{nesterov2008confidence} and \citet{nemirovski2009robust}, but with an extra orthogonal projection step that   is equivalent in our setting to  know a bound  on $\| \theta_\ast\|$, which is in practice not available.

\item[--] The constants in the  bounds of of Proposition~\ref{prop:boundp} (and thus other results as well) could clearly be improved. In particular, we have, for $p=1,2,3$ (see proof in Appendix~\ref{app:F}):
 \BEAS
  \E \Big(
2\gamma n  \big[  f ( \bar{\theta}_n) -   f(\theta^\ast)  \big]+\| \theta_n - \theta_\ast \|^{2} \Big)
&  \leqslant &     \|\theta_0 - \theta_\ast \|^2 + n   \gamma^2 R^2,
\\
 \E \Big(
2\gamma n  \big[  f ( \bar{\theta}_n) -   f(\theta^\ast)  \big]+\| \theta_n - \theta_\ast \|^{2} \Big)^2
& \leqslant & \big( \|\theta_0 - \theta_\ast \|^2 +  9 n  \gamma^2 R^2  \big)^{2},
\\
 \E \Big(
2\gamma n  \big[  f ( \bar{\theta}_n) -   f(\theta^\ast)  \big]+\| \theta_n - \theta_\ast \|^{2} \Big)^3
& \leqslant & \big( \|\theta_0 - \theta_\ast \|^2 +  20 n  \gamma^2 R^2  \big)^{3}.
 \EEAS

\EIT 

\subsection{Convergence of Gradients}
\label{sec:grad}

In this section, we prove higher-order bounds on the convergence of the gradient, with an improved rate $O(1/n)$ for $\| f'(\bar{\theta}_n ) \|^2$. In this section, we will need  the self-concordance property in Assumption~\textbf{(A2)}.
\begin{proposition}
\label{prop:grad}
Assume   \textbf{(A1-7)}.
With constant step size equal to $\gamma$, for any $n \geqslant 0$ and integer $p$, we have:
{\small
$$  
\!\!
\bigg( \E \bigg\|   f' \bigg(\frac{1}{n} \sum_{k=1}^n \theta_{k-1}   \bigg) \bigg\|^{2p} \bigg)^{1/2p} 
 \!\! \leqslant  \frac{R}{\sqrt{n}}
\bigg[
 {8}{\sqrt{p}} + \frac{4 p}{\sqrt{n}} + 40  R^2 \gamma p \sqrt{n}
 +
\frac{3}{ \gamma  \sqrt{n}}
\|\theta_0 - \theta_\ast\|^2  + \frac{3}{\gamma R \sqrt{n}} \| \theta_0 - \theta_\ast\|
\bigg].
$$
}
\end{proposition}
\begin{corollary}
Assume   \textbf{(A1-7)}.
With constant step size equal to $\gamma = \frac{1}{2R^2\sqrt{N}}$, for any integer $p $, we have:
$$
\bigg( \E \bigg\|   f' \bigg(\frac{1}{N} \sum_{k=1}^N \theta_{k-1}   \bigg) \bigg\|^{2p} \bigg)^{1/2p} 
 \!\! \leqslant  \!
\frac{R}{\sqrt{N}}
\bigg[
 {8}{\sqrt{p}} + \frac{4 p}{\sqrt{n}} + 20  p
 +
{6R^2}
\|\theta_0 - \theta_\ast\|^2  +  6R \| \theta_0 - \theta_\ast\|
\bigg].
 $$
\end{corollary}

 We can make the following observations:
\BIT
\item[--] The squared norm of the gradient $\|f'(\bar{\theta}_N)\|^2$ converges at rate $O(1/N)$.
\item[--] Given that $(\E \|f'(\bar{\theta}_N)\|^{2p})^{1/2p}$ is affine in $p$, we obtain a subexponential behavior for $\|f'(\bar{\theta}_N)\|$, that is, tail bounds similar to an exponential distribution.
\item[--] The proof of Proposition~\ref{prop:grad} makes use of the self-concordance assumption (that allows to upperbound deviations of gradients by deviations of function values) together with the proof technique of \citet{polyak1992acceleration}.

\EIT 

 \section{Self-Concordance Analysis for Strongly-Convex Problems}
 \label{sec:self}
 \label{sec:strong}

In the previous section, we have shown that $\| f'(\bar{\theta}_N) \|^2$ is of order $O(1/N)$. If the function $f$ was strongly convex with constant $\mu >0$, this would immediately lead to the bound $  f(\bar{\theta}_N) - f(\theta_\ast) \leqslant \frac{1}{2 \mu} \| f'(\bar{\theta}_N) \|^2$, of order $O(1/\mu N)$. However, because of the Lipschitz-continuity of $f$ on the full Hilbert space $\H$, it cannot be strongly convex.
In this section, we show how the self-concordance assumption may be used to obtain the exact same behavior, but with $\mu$ replaced by the \emph{local} strong convexity constant, which is more likely to be strictly positive.

The required property is summarized in the following proposition about (generalized) self-concordant function (see proof in Appendix~\ref{app:selfc}):
\begin{lemma}
\label{prop:selfc}
Let $f$ be a convex three-times differentiable function from $\H$ to $\rb$, such that for all $\theta_1,\theta_2 \in \H$, the function $\varphi: t \mapsto f\big[\theta_1 + t (\theta_2 - \theta_1) \big]$ satisfies: $\forall t \in \rb$, $| \varphi'''(t) |\leqslant R \| \theta_1 - \theta_2 \| \varphi''(t)$. Let $\theta_\ast$ be a global minimizer of $f$ and $\mu$ the lowest eigenvalue of $f''(\theta_\ast)$, which is assumed strictly positive.
$$
\mbox{If }  \frac{\| f'(\theta) \| R  }{\mu} \leqslant \frac{3}{4} \mbox{ , then }
\| \theta - \theta_\ast \|^2 \leqslant  4  \frac{\| f'(\theta) \|^2 }{\mu^2}
\mbox{ and } f(\theta) - f(\theta_\ast)  \leqslant 2 \frac{\| f'(\theta) \|^2  }{\mu  }.
$$
\end{lemma}

 We may now use this proposition for the averaged stochastic gradient. For simplicity, we only consider the step-size $\gamma = \frac{1}{2R^2 \sqrt{N}}$, and the last iterate (see proof in Appendix~\ref{sec:boundself}):
  \begin{proposition}
  \label{prop:boundself}
  Assume   \textbf{(A1-7)}.
  Assume  $\gamma = \frac{1}{2R^2 \sqrt{N}}$. Let $\mu > 0 $ be the lowest eigenvalue of the Hessian of $f$ at the unique global optimum $\theta_\ast$. Then:
  \BEAS
  \E f(\bar{\theta}_N) - f(\theta_\ast) & \leqslant &   \frac{  R^2}{N\mu}
\Big(   5R 
\|\theta_0 - \theta_\ast\| +  15  \Big)^4 ,\\
  \E  \big\| \bar{\theta}_N - \theta_\ast \big\|^2 & \leqslant &   \frac{  R^2}{N\mu^2}
\Big(   6 R 
\|\theta_0 - \theta_\ast\| +  21  \Big)^4 .
\EEAS
   \end{proposition}

  We can make the following observations:
  \BIT
  \item[--] The proof relies on Lemma~\ref{prop:selfc} and requires a control of the probability that $ \frac{\| f'(\bar{\theta}_N) \| R  }{\mu} \leqslant \frac{3}{4}$, which is obtained from Proposition~\ref{prop:grad}.
  \item[--] We conjecture a bound of the form $ \Big[ \frac{R^2}{N\mu} ( \square R\|\theta_0 - \theta_\ast\| + \triangle\sqrt{p} )^4   \Big]^p$ for the  $p$-th order moment of $f(\bar{\theta}_N) - f(\theta_\ast) $, for some scalar constants $\square$ and $\triangle$.
  \item[--] The new bound now has the term $R 
\|\theta_0 - \theta_\ast\|$ with a fourth power (compared to the bound in Lemma~\ref{prop:old}, which has a second power), which typically grows with the dimension of the underlying space (or the slowness of the decay of eigenvalues of the covariance operator when $\H$ is infinite-dimensional). It would be interesting to study whether this dependence can be reduced.
\item[--] The key elements in the previous proposition are that  (a) the constant $\mu$ is the \emph{local} convexity constant, and (b) the step-size does not depend on that constant $\mu$, hence the claimed adaptivity.
\item[--] The bounds are only better than the non-strongly-convex bounds from Lemma~\ref{prop:old}, when the Hessian lowest eigenvalue is large enough, that is,  $\mu R^2  \sqrt{N}$ larger than a fixed constant.
\item[--] In the context of logistic regression, even when the covariance matrix of the inputs is invertible, then the only available lower bound on $\mu$ is equal to the lowest eigenvalue of the covariance matrix times  $\exp(- R \| \theta_\ast\| )$, which is exponentially small. However, the previous bound is overly pessimistic since it is based on an upper bound on the largest possible value of $\langle x , \theta_\ast \rangle$. In practice, the actual value of $\mu$ is much larger and only a small constant smaller than the lowest eigenvalue of the covariance matrix.
 In order to assess if this result can be improved, it is interesting to look at the asymptotic result from \citet{polyak1992acceleration} for logistic regression, which leads to a limit rate of $1/n$ times $\tr f''(\theta_\ast)^{-1} \big( \E f_n'(\theta_\ast)  f_n'(\theta_\ast) ^\top\big)$; note that this rate holds both for the stochastic approximation algorithm and for the global optimum of the training cost, using standard asymptotic statistics results \citep{VanDerVaart}. 
 When the model is well-specified, that is, the log-odds ratio of the conditional distribution of the label given the input is linear, then  
$ \E f_n'(\theta_\ast)  f_n'(\theta_\ast) ^\top = \E f_n''(\theta_\ast) = f''(\theta_\ast)$, and the asymptotic rate is exactly $d/n$, where $d$ is the dimension of $\H$ (which has to be finite-dimensional for the covariance matrix to be invertible). It would be interesting to see if making the extra assumption of well-specification, we can also get an improved \emph{non-asymptotic} result. When the model is mis-specified however, the quantity $ \E f_n'(\theta_\ast)  f_n'(\theta_\ast) ^\top$ may be large even when $f''(\theta_\ast)$ is small, and the asymptotic regime does not readily lead to an improved bound.
\EIT

 \section{Conclusion}

In this paper, we have provided a novel analysis of averaged stochastic gradient for logistic regression and related problems. The key aspects of our result are (a) the adaptivity to local strong convexity provided by averaging and (b) the use of self-concordance to obtain a simple bound that does not involve a term which is explicitly exponential in $R \|\theta_0 - \theta_\ast\|$, which could be obtained by constraining the domain of the iterates.

Our results could be extended in several ways: (a) with a finite and known horizon $N$, we considered a constant step-size proportional to $1/R^2\sqrt{N}$; it thus seems natural to study the decaying step size $\gamma_n = O (  {1}/{R^2 \sqrt{n}})$, which should, up to logarithmic terms, lead to similar results---and thus likely provide a solution to a a recently posed open problem for online logistic regression \citep{open}; (b) an alternative would be to consider a doubling trick where the step-sizes are piecewise constant; also, (c) it may be possible to consider other assumptions, such as exp-concavity \citep{hazanbeyond} or uniform convexity \citep{juditsky2010primal}, to derive similar or improved results. Finally, by departing from a plain averaged stochastic gradient recursion, \citet{newsto} have considered an online Newton algorithm with the same running-time complexity, which leads to a rate of $O(1/n)$ without strong convexity assumptions for logistic regression (though with additional assumptions regarding the distributions of the inputs). It would be interesting to understand if simple assumptions such as the ones made in the present paper are possible while preserving the improved convergence rate.

\acks{The author was partially supported by the European Research Council (SIERRA Project), and thanks
Simon Lacoste-Julien, Eric Moulines and Mark Schmidt for helpful discussions. Morever, Alekh Agarwal suggested and provided a detailed outline of the proof technique based on Freedman's inequality; this was greatly appreciated.}

 \appendix
 
 \section{Probability Lemmas}
 \label{app:proba}
  In this appendix, we prove simple lemmas relating bounds on moments to tail bounds, with the traditional use of Markov's inequality. See more general results by \citet{boucheron2013concentration}.

 \begin{lemma}
 \label{lemma:conc}
 Let $X$ be a non-negative random variable such that for some positive constants $A$ and $B$, and all $p \in \{1,\dots,n\}$,
 $$\E X^p \leqslant ( A + Bp)^p.$$ Then, if $ t \leqslant \frac{n}{2}$,
 $$
\P( X \geqslant 3B t + 2 A  ) \leqslant  2 \exp( -  t ).
 $$
 \end{lemma}
 \begin{proof}
We have, by Markov's inequality, for any $p \in \{1,\dots,n\}$:
 \BEAS
 \P( X \geqslant 2Bp + 2 A  ) \leqslant \frac{ \E X^p }{(2Bp + 2 A)^p} \leqslant \frac{ ( A  + Bp)^p}{(2A +2Bp)^p}
 =   \exp( - \log(2) p  ).
\EEAS
For $u \in [1,n]$, we consider $p = \lfloor u\rfloor$, so that 
$$
\P( X \geqslant 2Bu + 2 A  ) \leqslant
\P( X \geqslant 2Bp + 2 A  ) \leqslant \exp( - \log(2) p  ) \leqslant 2 \exp( - \log(2) u  ).
$$
We take $t = \log(2) u $ and use $2 / \log 2 \leqslant 3$.
 This is thus valid if $  t \leqslant \frac{n}{2}$.
  \end{proof}

 \begin{lemma}
 \label{lemma:conc2}
 Let $X$ be a non-negative random variable such that  for some positive constants $A$,  $B$ and $C$, and  for all $p \in \{1,\dots,n\}$,
 $$\E X^p \leqslant ( A\sqrt{p} + Bp + C)^{2p}.$$ Then, if $ t \leqslant n$,
 $$
\P( X   \geqslant (2A\sqrt{t} + 2Bt + 2C)^2 ) \leqslant  4 \exp( -  t ).
 $$
 \end{lemma}
 \begin{proof}
We have, by Markov's inequality, for any $p \in \{1,\dots,n\}$:
 \BEAS
 \P( X \geqslant (2A\sqrt{p} + 2Bp + 2C)^2  ) & \leqslant & \frac{ \E X^p }{(2A\sqrt{p} + 2Bp + 2C)^{2p}}
 \\ & \leqslant & \frac{ ( A\sqrt{p} + Bp + C)^{2p}}{(2A\sqrt{p} + 2Bp + 2C)^{2p}}
 \leqslant   \exp( - \log(4) p  ).
\EEAS
For $u \in [1,n]$, we consider $p = \lfloor u\rfloor$, so that 
\BEAS
\P( X \geqslant (2A\sqrt{u} + 2Bu + 2C)^2  ) &  \leqslant & 
\P( X \geqslant (2A\sqrt{u} + 2Bu + 2C)^2  ) 
\\
& \leqslant  & \exp( - \log(2) p  ) \leqslant 4 \exp( - \log(4) u  ).
\EEAS
We take $t = \log(4) u $ and use $ \log 4 \geqslant 1$.
 This is thus valid if $  t \leqslant n	$.
  \end{proof}

  \section{Self-Concordance Properties}
  
  In this appendix, we show two lemmas regarding our generalized notion of self-concordance, as well as   Lemma~\ref{prop:selfc}. For more details, see \citet{bach2010self} and references therein.
  
  The following lemma provide an upper-bound on a one-dimensional self-concordant function at a given point which is based on the gradient at this point and the value and the Hessian at the global minimum. This is key to going in \mysec{strong} from a convergence of gradients to a convergence of function values.
  
  \begin{lemma}
  \label{lemma:selfc}

Let $\varphi:[0,1] \to \rb$ a strictly  convex three-times differentiable function such that for some $S >0$,  $\forall t \in [0,1], \ |\varphi'''(t)| \leqslant S \varphi''(t)$. Assume $\varphi'(0)=0$, $\varphi''(0)>0$. Then:
$$
\frac{\varphi'(1)}{\varphi''(0)} S \geqslant 1 - e^{-S} \mbox{ and } 
\varphi(1)  \leqslant  \varphi(0) + \frac{  \varphi'(1)  ^2 }{\varphi''(0) } ( 1 + S) .
$$
Moreover, if $\alpha = \frac{\varphi'(1) S }{\varphi''(0)} < 1$, then $\displaystyle \varphi(1)  \leqslant  \varphi(0) +\frac{  \varphi'(1)  ^2 }{\varphi''(0) } \frac{1}{\alpha} \log \frac{1}{1-\alpha}$. If in addition $\alpha \leqslant \frac{3}{4}$,  then
$\varphi(1)  \leqslant  \varphi(0) + 2 \frac{  \varphi'(1)  ^2 }{\varphi''(0) }$ and
$\varphi''(0) \leqslant 2 \varphi'(1)$.

\end{lemma}
\begin{proof}
By self-concordance, we obtain that the derivative of $u \mapsto \log \varphi''(u)$ is lower-bounded by~$-S$. By integrating between $0$ and $t \in [0,1]$,
we get
\BEQ
\label{eq:ABA}
\log \varphi''(t) - \log \varphi''(0) \geqslant - S t \mbox{ , that is, } 
\varphi''(t) \geqslant \varphi''(0) e^{-St},
\EEQ
and by integrating between $0$ and $1$, we obtain (note that we have assumed $\varphi'(0)=0$):
\BEQ
\label{eq:AA}
\varphi'(1) \geqslant \varphi''(0)  \frac{ 1 - e^{-S}}{S}.
\EEQ
We then get (with a first inequality from convexity of $\varphi$, and the last inequality from $e^S \geqslant 1 + S$):
$$
\varphi(1) - \varphi(0) \leqslant \varphi'(1)
\leqslant  \varphi'(1) \frac{\varphi'(1)}{\varphi''(0) } \frac{S}{ 1 - e^{-S}}
=  \frac{  \varphi'(1)  ^2 }{\varphi''(0) } \bigg( S + \frac{S}{e^S-1} \bigg)
\leqslant  \frac{  \varphi'(1)  ^2 }{\varphi''(0) } ( 1 + S).
$$
\eq{AA} implies that $\alpha \geqslant 1 - e^{-S}$, which implies, if $\alpha < 1$, $S \leqslant \log \frac{1}{1-\alpha}$.  This implies that
$$
\varphi(1) - \varphi(0) \leqslant \varphi'(1) \frac{\varphi'(1)}{\varphi''(0) } \frac{S}{ 1 - e^{-S}} \leqslant  \frac{\varphi'(1)^2}{\varphi''(0) }  \frac{1}{\alpha} \log \frac{1}{1-\alpha},
$$
using the monotonicity of $S \mapsto \frac{S}{ 1 - e^{-S}} $.
Finally the last bounds are a consequence of
$\frac{S}{\alpha} \leqslant \frac{1}{\alpha} \log \frac{1}{1-\alpha} \leqslant 2$, which is valid for $\alpha \leqslant \frac{3}{4}$.

Note that in \eq{ABA}, we do consider a lower-bound on the Hessian with an exponential factor $e^{-St}$. The key feature of using self-concordance properties is to get around this exponential factor in the final bound.
\end{proof}

The following lemma upper-bounds the remainder in the first-order Taylor expansion of the gradient by the remainder in the first-order Taylor expansion of the function. This is important when function values behave well (i.e., converge to the minimal value) while the iterates may not.

 \begin{lemma}
  \label{lemma:selfcontrol}
Let $f$ be a convex three-times differentiable function from $\H$ to $\rb$, such that for all $\theta_1,\theta_2 \in \H$, the function $\varphi: t \mapsto f\big[\theta_1 + t (\theta_2 - \theta_1) \big]$ satisfies: $\forall t \in \rb$, $| \varphi'''(t) |\leqslant R \| \theta_1 - \theta_2 \| \varphi''(t)$. For any $\theta_1,\theta_2 \in H$, we have:
$$
\big\|
f'(\theta_1)- f'(\theta_2) -  f''(\theta_2) ( \theta_2 - \theta_1 )
\big\| \leqslant R \big[
f(\theta_1)- f(\theta_2) -  \langle f'(\theta_2) , \theta_2 - \theta_1 \rangle
\big]
.$$
\end{lemma}
 \begin{proof}
 For a given $z \in \H$ of unit norm,
 let $\varphi(t) = \big\langle z, f'\big(\theta_2 + t( \theta_1 - \theta_2) \big)-   f'(\theta_2) -  t  f''(\theta_2) ( \theta_2 - \theta_1 )\big\rangle$
 and $\psi(t) = R \big[
f( \theta_2 + t( \theta_1 - \theta_2) )- f(\theta_2) -   t\langle f'(\theta_2) , \theta_2 - \theta_1 \rangle
\big]
$. We have $\varphi(0) = \psi(0) = 0$. Moreover, we have the following derivatives:
\BEAS
\varphi'(t) & \!\!=\!\! & \big\langle z, f''\big(\theta_2 + t( \theta_1 - \theta_2)\big) - f''(\theta_2), \theta_1 - \theta_2 \big\rangle \\
\varphi''(t) & \!\!=\!\! &  f'''\big(\theta_2 + t( \theta_1 - \theta_2)\big)[ z, \theta_1-\theta_2, \theta_1 - \theta_2] \\
  & \leqslant & R\| z\|_2   f''\big(\theta_2 + t( \theta_1 - \theta_2)\big)[\theta_1-\theta_2, \theta_1 - \theta_2], \mbox{ using the Appendix A of \citet{bach2010self},} \\
  &  \!\!=\!\!  & R \big\langle   \theta_2 - \theta_1, f''\big(\theta_2 + t( \theta_1 - \theta_2)\big) (\theta_1 - \theta_2 ) \big\rangle \\
\psi'(t) &\!\!=\!\! & R \big\langle   f'\big(\theta_2 + t( \theta_1 - \theta_2)\big) - f'(\theta_2), \theta_1 - \theta_2 \big\rangle \\
\psi''(t) & \!\!=\!\! & R \big\langle   \theta_2 - \theta_1, f''\big(\theta_2 + t( \theta_1 - \theta_2)\big) (\theta_1 - \theta_2 ) \big\rangle,
\EEAS 
where $f'''(\theta)$ is the third order tensor of third derivatives.
This leads to $\varphi'(0) = \psi'(0) = 0$ and $\varphi''(t) \leqslant \psi''(t)$. We thus have $\varphi(1) \leqslant \psi(1)$ by integrating twice, which leads to the desired result by maximizing with respect to~$z$.
 \end{proof}
 
 \subsection{Proof of Lemma~\ref{prop:selfc}}
\label{app:selfc}
We follow the standard proof techniques in self-concordant analysis and define an appropriate function of a single real variable and apply simple lemmas like the ones above.

Define $\varphi: t \mapsto f\big[\theta_\ast + t (\theta - \theta_\ast) \big] - f(\theta_\ast)$. We have
\BEAS
\varphi'(t)  & = &  \big\langle  f'\big[\theta_\ast + t (\theta - \theta_\ast) \big], \theta - \theta_\ast \rangle \\
\varphi''(t) & = &  \big\langle \theta - \theta_\ast,  f''\big[\theta_\ast + t (\theta - \theta_\ast) \big] (\theta - \theta_\ast) \rangle\\
\varphi'''(t) & = &   f'''\big[\theta_\ast + t (\theta - \theta_\ast) \big] [\theta - \theta_\ast,\theta - \theta_\ast,\theta - \theta_\ast].
\EEAS
We thus have: $\varphi(0) = \varphi'(0) = 0$, $0 \leqslant \varphi'(1) = \langle f'(\theta), \theta - \theta_\ast \rangle \leqslant \| f'(\theta) \| \| \theta - \theta_\ast\|$, $\varphi''(0) = \langle \theta - \theta_\ast ,  f''(\theta_\ast) ( \theta - \theta_\ast) \rangle \geqslant \mu \| \theta - \theta_\ast\|^2$,
and $\varphi(t) \geqslant 0$ for all $t \in [0,1]$. Moreover,  $\varphi'''(t) \leqslant R \| \theta - \theta_\ast \| \varphi''(t)$ for all $t \in [0,1]$, that is,  Lemma~\ref{lemma:selfc} applies with $S = R \| \theta - \theta_\ast\|$. This  leads to the desired result, with $\alpha = 
\frac{\varphi'(1) S}{ \varphi''(0) } \leqslant \frac{ \| f'(\theta) \| R}{\mu}$. Note that we also have (using the second inequality in  Lemma~\ref{lemma:selfc}), for all $\theta \in \H$ (and without any assumption on $\theta$):
$$
f(\theta) - f(\theta_\ast)  \leqslant  \big( 1 + R \| \theta - \theta_\ast \|
\big) \frac{\| f'(\theta) \|^2  }{\mu  } .
$$

\section{Proof of Proposition~\ref{prop:boundp}}
 \label{app:boundp}
 
 We provide two alternative proofs of the same result: (a) our original somewhat tedious proof in 
Appendices~\ref{app:mine1} and \ref{app:mine2},
 based on taking powers of the inequality in \eq{mn} and using martingale moment inequalities, (b) a shorter proof in \myapp{newsto}, later derived by \citet{newsto}, that uses Burkholder-Rosenthal-Pinelis  inequality \citep[Theorem 4.1]{pinelis}. Another  proof technique was  suggested and  outlined by Alekh Agarwal (personal communication), that uses 
Freedman's inequality for martingales \citep[Theorem 1.6]{freedman}; it allows to directly get a tail bound like in Proposition~\ref{prop:dev}. This proof will be presented in \myapp{suggested}.

Note that the two shorter proofs currently lead to slightly worse constants  (or to extra logarithmic factors), that may be improved with more refined derivations.

All proofs start from a similar martingale set-up that we describe in \myapp{mart} and use an almost-sure bound when $p$ gets large (\myapp{small}).

 \subsection{Bounding Martingales}
 \label{app:mart}
From the proof of Lemma~\ref{prop:old}, we have the recursion:
\BEAS
2 \gamma\big[  f(\theta_{n-1}) - f(\theta_\ast)  \big] +   \| \theta_{n }- \theta_\ast \|^2
& \leqslant &   \| \theta_{n-1 }- \theta_\ast \|^2
+ \gamma^2   R^2 + M_n,
\EEAS
     with 
      $$
      M_n =  -  2 \gamma  \langle
 \theta_{n-1} - \theta_\ast ,  f_n'(\theta_{n-1}) -  f'(\theta_{n-1}) \rangle.$$
 This leads to, by summing from $1$ to $n$, and using the convexity of $f$:
$$ 2\gamma n  f \bigg( \frac{1}{n} \sum_{k=1}^{n} \theta_{k-1} \bigg) - 2\gamma n f(\theta^\ast)  +\| \theta_n - \theta_\ast \|^{2}  \leqslant A_n,$$
with
$$
A_n =  \| \theta_{0 }- \theta_\ast \|^2
+ n \gamma^2   R^2 + \sum_{k=1}^n M_k
 \geqslant 0. 
$$
 Note that $A_n$ may also be defined recursively as $A_0 = \| \theta_0 - \theta_\ast\|^2$ and 
\BEQ
\label{eq:An}
A_n = A_{n-1}
+ \gamma^2   R^2 + M_n.
\EEQ

 The random variables $(M_n)$ and $(A_n)$ satisfy the following properties that will proved useful throughout the proof:
 \BIT
 \item[(a)] Martingale increment: for all $k \geqslant 1$, $\E(M_k | \F_{k-1})=0$. This implies that $S_n = \sum_{k=1}^n M_k$ is a martingale.
 \item[(b)]
 Boundedness: 
 $
 |M_k|   \leqslant    4 \gamma R \| \theta_{k-1} - \theta_\ast\| \leqslant 4\gamma R A_{k-1}^{1/2}
$ almost surely.
 \EIT

\subsection{Almost Sure Bound}
\label{app:small}
In this section, we derive an almost sure bound that will be valid for small $n$. From the stochastic gradient recursion $\theta_n = \theta_{n-1} - \gamma f_n'(\theta_{n-1})$, we get, using Assumption \textbf{(A4)} and the triangle inequality:
$$
\| \theta_n - \theta_\ast\|   \leqslant   
\| \theta_{n-1} - \theta_\ast\| + \gamma \| f_n'(\theta_{n-1}) \| \leqslant   \| \theta_{n-1} - \theta_\ast\| + \gamma R \ \mbox{ almost surely}.
$$
This leads to $\| \theta_n - \theta_\ast\|  \leqslant \| \theta_{0} - \theta_\ast\| + n \gamma R$ for all $n \geqslant 0$. This in turn implies that
\BEA
\nonumber A_n & \leqslant & \| \theta_0 - \theta_\ast\|^2 + n \gamma^2 R^2 + 4 \gamma R \sum_{k=1}^n \| \theta_{k-1} - \theta_\ast\| \ \mbox{ using } |M_k|   \leqslant    4 \gamma R \| \theta_{k-1} - \theta_\ast\| ,\\[-.05cm]
\nonumber   & \leqslant & \| \theta_0 - \theta_\ast\|^2 + n \gamma^2 R^2 + 4 \gamma R  \sum_{k=1}^n \big[ \| \theta_{0} - \theta_\ast\| + (k-1)\gamma R \big] \mbox{ using the inequality above,}\\
\nonumber & \leqslant & \| \theta_0 - \theta_\ast\|^2 + n \gamma^2 R^2 + 4 \gamma n R   \| \theta_{0} - \theta_\ast\| + 2\gamma^2 R^2 n^2 \\
\nonumber & & \hspace*{7cm} \mbox{ by summing over the first } n-1 \mbox{ integers,}\\
 \nonumber & \leqslant & \| \theta_0 - \theta_\ast\|^2 + n \gamma^2 R^2 + 2 \gamma^2 n^2 R^2  + 2   \| \theta_{0} - \theta_\ast\|^2 + 2\gamma^2 R^2 n^2  \mbox{ using } ab \leqslant \frac{a^2}{2} + \frac{b^2}{2} ,\\
\label{eq:always}  & \leqslant & 3 \| \theta_0 - \theta_\ast\|^2 + 5 n^2 \gamma^2 R^2 
\mbox{ almost surely.}
\EEA
This implies that the bound is shown for all $p \geqslant \frac{n}{4}$.

\subsection{Derivation of $p$-th Order Recursion}
\label{app:mine1}
 
 The first proof works as follows: (a) derive a recursion between the $p$-th moments and the lower-order moments (this section) and (c) prove the result by induction on $p$ (\myapp{mine2}). Note that we have to treat separately small values on $n$ in the recursion, for which we use the almost sure bound from Appendix~\ref{app:small}.

Starting from \eq{An}, using the binomial expansion formula, we get:
\BEAS   A_n^{p} 
& \!\!\!\! \leqslant \!\!\!\!  &  \Big( A_{n-1}
+ \gamma^2   R^2 + M_n\Big)^p
 =    \sum_{k=0}^p  { p \choose k} \big(A_{n-1} + \gamma^2 R^2 \big) ^{p-k}M_n^k \\
 &\!\!\!\!  \leqslant \!\!\!\! &  \big(A_{n-1} + \gamma^2 R^2 \big) ^{p}
 + p \big(A_{n-1} + \gamma^2 R^2 \big) ^{p-1}M_n
 + \sum_{k=2}^p  { p \choose k} \big(A_{n-1} + \gamma^2 R^2 \big) ^{p-k}\big(  4 \gamma R A_{n-1}^{1/2}  \big)^k  .
\EEAS
This leads to, using $E(M_n|\F_{n-1})=0$, upper bounding $\gamma^2 R^2$ by $4 \gamma^2 R^2$, and using the binomial expansion formula several times:
\BEAS
\E \big [   A_{n}^p \big| \F_{n-1} \big] &
 \leqslant &  
 \big( A_{n-1}  + 4\gamma^2 R^2 \big) ^{p}  
 + \sum_{k=2}^p  { p \choose k} \big( A_{n-1}  + 4\gamma^2 R^2 \big) ^{p-k} \big(  4 \gamma R A_{n-1}^{1/2}  \big)^k\\
  &
 = &  
 \big( A_{n-1}  + 4\gamma^2 R^2 +    4 \gamma RA_{n-1}^{1/2} \big) ^{p}  - 4 \gamma R p \big( A_{n-1} + 4 \gamma^2 R^2 \big) ^{p-1}  A_{n-1} ^{1/2}\\
 & & \hspace*{2cm} \mbox{ by isolating the term $k=1$ in the binomial formula,}\\
 &
 = &  
 \big( A_{n-1} ^{1/2}  + 2\gamma R  \big) ^{2p}  - 4 \gamma R p \big( A_{n-1} + 4\gamma^2 R^2 \big) ^{p-1}  A_{n-1} ^{1/2} \\
 & = & \sum_{k=0}^{2p} { 2p \choose k} A_{n-1}^{k/2} ( 2\gamma R)^{2p-k}  - 4 \gamma R p    A_{n-1}^{1/2} \sum_{k=0}^{p-1} { p-1 \choose k} A_{n-1} ^k  ( 2\gamma R)^{2(p-1-k)} \\
 & = &  \sum_{k=0}^{2p} A_{n-1} ^{k/2} ( 2\gamma R)^{2p-k } C_{k},
\EEAS
with the constants $C_k$ defined as:
\BEAS
C_{2q} & = & { 2p \choose 2q} \mbox{ for } q  \in\{0,\dots, p\}, \\
C_{2q+1} & = & { 2p \choose 2q+1} - 2p { p-1 \choose q }  \mbox{ for } q  \in\{0,\dots, p-1\}.
\EEAS
In particular, $C_0=1$, $C_{2p}=1$,  $C_1=0$ and $C_{2p-1}= { 2p \choose 2p-1} - 2p { p-1 \choose p-1 } = 0$.

Our goal is now to bounding the values of $C_k$ to obtain \eq{qwe} below. This will be done by bounding the odd-indexed element by the even-indexed elements.

We have, for $q  \in\{1,\dots, p-2\} $,
\BEA
\nonumber C_{2q+1} \frac{ 2q+1}{2p-2q-1} 
& \leqslant &  { 2p \choose 2q+1}  \frac{ 2q+1}{2p-2q-1} \\
\nonumber& = & \frac{ (2p)! }{ (2q+1)! ( 2p-2q-1)!} \frac{ 2q+1}{2p-2q-1} 
\\
\label{eq:2p}& = & \frac{ (2p)! }{ (2q)! ( 2p-2q)!} \frac{ 2p - 2q}{2p-2q-1} 
=  { 2p \choose 2q } \frac{ 2p - 2q}{2p-2q-1} .
\EEA
For the end of the   interval above in $q$, that is, $q=p-2$, we obtain $C_{2q+1} \frac{ 2q+1}{2p-2q-1} \leqslant C_{2q} \frac{4}{3}$,
while for $q\leqslant p-3$, we obtain $C_{2q+1} \frac{ 2q+1}{2p-2q-1} \leqslant C_{2q} \frac{6}{5}$.

Moreover, for $q  \in\{1,\dots, p-2\} $,
\BEA
\nonumber C_{2q+1} \frac{2p-2q-1} { 2q+1}
& \leqslant &  { 2p \choose 2q+1}  \frac{2p-2q-1}{ 2q+1} \\
\nonumber& = & \frac{ (2p)! }{ (2q+1)! ( 2p-2q-1)!} \frac{2p-2q-1} { 2q+1}
\\
\label{eq:2p2}& = & \frac{ (2p)! }{ (2q+2)! ( 2p-2q-2)!} \frac{ 2q+2}{2q+1} 
=  { 2p \choose 2q+2 } \frac{ 2q+2}{2q+1}  .
\EEA
For the end of the   interval above in $q$, that is,  $q=1$, we obtain $C_{2q+1} \frac{2p-2q-1} { 2q+1} \leqslant C_{2q+2} \frac{4}{3} $, while for $q\geqslant 2$, we obtain $C_{2q+1} \frac{2p-2q-1} { 2q+1} \leqslant C_{2q+2} \frac{6}{5} $.

We have moreover, by using the bound $  2 \gamma R A_{n-1}^{1/2} \leqslant \frac{\alpha}{2}  (2 \gamma R)^2 + \frac{1}{2\alpha} A_{n-1}$ for $\alpha = \frac{2q+1}{2p-2q-1}$:
\BEAS
& & C_{2q+1} A_{n-1} ^{q+1/2} (2\gamma R)^{2p - 2q - 1}  \\
& = &  C_{2q+1} A_{n-1} ^{q} (2\gamma R)^{2p - 2q - 2}  A_{n-1}^{1/2} ( 2 \gamma R) \\
& \leqslant &  C_{2q+1} A_{n-1} ^{q} (2\gamma R)^{2p - 2q - 2} \frac{1}{2}\bigg[
\frac{2q+1}{2p-2q-1} ( 2 \gamma R)^2   + \frac{2p-2q-1} {2q+1} A_{n-1}
\bigg]  \\
& = & \frac{1}{2}  C_{2q+1} \frac{2p-2q-1}{2q+1}   A_{n-1} ^{q+1}  (2\gamma R)^{2p - 2q - 2}
+\frac{1}{2}  C_{2q+1} \frac{2q+1}{2p-2q-1}   A_{n-1} ^{q}(2\gamma R)^{2p - 2q } .
\EEAS
By combining the previous inequality with \eq{2p} and \eq{2p2}, we get that the terms indexed by $2q+1$ are bounded by the terms indexed by $2q+2 $ and $2q$. All terms with $ q \in \{2,\dots,p-3\}$ are expanded with constants $\frac{3}{5}$, while for $q=1$ and $q=p-2$, this is $\frac{2}{3}$. Overall each even term receives a contribution which is less than
$\max\{  \frac{6}{5}, \frac{3}{5} + \frac{2}{3} , \frac{2}{3}\} = \frac{19}{15}$.
This leads to
\BEAS
  \sum_{q=1}^{p-2}
C_{2q+1} A_{n-1} ^{q+1/2} (2\gamma R)^{2p - 2q - 1}   
&  
  \leqslant  &  \frac{19}{15}
\sum_{q=0}^{p-1}
C_{2q} A_{n-1} ^{q} (2\gamma R)^{2p - 2q },
\EEAS
leading to the recursion that will allow us to derive our result:
\BEQ
\label{eq:qwe}
\E \big [   A_{n} ^{p}  \big| \F_{n-1} \big]  
 \leqslant     A_{n-1}^{p}  
 +
\frac{34}{15}
\sum_{q=0}^{p-1}
{ 2p \choose 2q} A_{n-1} ^{q} (2\gamma R)^{2p - 2q }.
 \EEQ

\subsection{Proof by Induction}
\label{app:mine2}
 
We now proceed by induction on $p$.
If we assume that $ \E A_k^{q} \leqslant \big( 3 \|\theta_0 - \theta_\ast \|^2 + k q \gamma^2 R^2 B \big)^q$ for all $q < p$, and a certain $B$ (which we will choose to be equal to $20$). We first note that if $n\leqslant 4p$, then from \eq{always}, we have
\BEAS
\E {A_n^p}&  \leqslant &   \big( 3 \| \theta_0 - \theta_\ast\|^2 + 5 n^2 \gamma^2 R^2 \
\big)^p\\
&  \leqslant &   \big( 3 \| \theta_0 - \theta_\ast\|^2 + 20 n p \gamma^2 R^2 \
\big)^p.
\EEAS

Thus, we only need to consider $n \geqslant 4p$. We then get from \eq{qwe}:
\BEAS
\E \| \theta_{n }- \theta_\ast \|^{2p} & \leqslant &   \| \theta_{0 }- \theta_\ast \|^{2p} 
+ 
 \frac{34}{15} \sum_{k=0}^{n-1}
\sum_{q=0}^{p-1}
{ 2p \choose 2q} \E A_k^{q} (2\gamma R)^{2p - 2q }\\
& \leqslant &   \| \theta_{0 }- \theta_\ast \|^{2p} 
+ 
 \frac{34}{15} \sum_{k=0}^{n-1}
\sum_{q=0}^{p-1}
{ 2p \choose 2q} \big( 3\|\theta_0 - \theta_\ast \|^2 + k q \gamma^2 R^2 B \big)^q (2\gamma R)^{2p - 2q },
\EEAS
using the induction hypothesis. We may now sum with respect to $k$:
\BEAS
& &  \E \| \theta_{n }- \theta_\ast \|^{2p}  \\
& \!\!\leqslant \!\! &   \| \theta_{0 }- \theta_\ast \|^{2p} 
+ 
 \frac{34}{15} 
\sum_{q=0}^{p-1}
{ 2p \choose 2q} (2\gamma R)^{2p - 2q } \sum_{k=0}^{n-1} \big( 3 \|\theta_0 - \theta_\ast \|^2 + k q \gamma^2 R^2 B \big)^q 
\\
& \!\! \leqslant  \!\! &   \| \theta_{0 }- \theta_\ast \|^{2p} 
+ 
 \frac{34}{15} 
\sum_{q=0}^{p-1}
{ 2p \choose 2q} (2\gamma R)^{2p - 2q } \sum_{j=0}^q 3^j
\|\theta_0 - \theta_\ast \|^{2j } { q \choose j} \big( q \gamma^2 R^2 B \big)^{q-j}\frac{n^{q-j+1}}{q-j+1}
\\
& & \mbox{ using } \sum_{k=0}^{n-1}  k^\alpha \leqslant \frac{ n^{\alpha+1}}{\alpha + 1} \mbox{ for any } \alpha > 0 , \\
& \!\!= \!\! &   \| \theta_{0 }- \theta_\ast \|^{2p} 
+ 
 \frac{34}{15} \sum_{j=0}^{p-1}  3^j\|\theta_0 - \theta_\ast \|^{2j } (4\gamma^2 R^2 n)^{p-j}
\sum_{q=j}^{p-1}
{ 2p \choose 2q}  
 { q \choose j} \big( \frac{q   B}{4} \big)^{q-j}\frac{n^{q-p+1}}{q-j+1},
\EEAS
by changing the order of summations.
We now aim to show that it is less than
\BEAS
\Big( 3 \|\theta_0 - \theta_\ast \|^2 + k p \gamma^2 R^2 B \Big)^p
 =   3^p\| \theta_{0 }- \theta_\ast \|^{2p} 
 + \sum_{j=0}^{p-1} 3^j \|\theta_0 - \theta_\ast \|^{2j } (\gamma^2 R^2 n)^{p-j} (Bp)^{p-j} { p \choose j }.
\EEAS
By comparing all terms in $\| \theta_0 - \theta_\ast\|^{2j}$, this is true as soon as for all $ j \in \{0,\dots,p-1\}$,
$$
\frac{34}{15} \sum_{q=j}^{p-1}
{ 2p \choose 2q}  
 { q \choose j} \big( q   B /4\big)^{q-j}\frac{1}{q-j+1} \frac{1}{n^{p-q - 1}}
 \leqslant (Bp/4)^{p-j} { p \choose j }
 $$
 $$
 \Leftrightarrow 
\frac{34}{15} \sum_{k=0}^{p-1-j}
{ 2p \choose 2k+2}  
 { p\!- \! 1 \! - \! k \choose j} \big( (p-1-k )   B/4 \big)^{p-1-k -j}\frac{1}{p\!-\! k \! - \! j} \frac{1}{n^{k}}
 \leqslant (Bp/4)^{p-j} { p \choose j },
 $$
 obtained by using the change of variable $k = p - 1 - q$.
This is implied by, using $n \geqslant 4p$:
$$
\frac{136}{15}\sum_{k=0}^{p-1-j} B^{-1-k  } p^{-k-p+j}
{ 2p \choose 2k+2}  
 \frac{ { p-1-k \choose j}}{{ p \choose j }} \big(  p-1-k       \big)^{p-1-k -j}\frac{1}{p-k -j} 
  \leqslant    1.
$$
By expanding the binomial coefficients and simplifying by $p-k-j$, this is equivalent to
$$
  \frac{136}{15}\sum_{k=0}^{p-1-j}B^{-1-k  } p^{-k-p+j}
{ 2p \choose 2k+2}  
 \frac{  (p-1-k) \cdots (p-k-j+1) }{p \cdots (p-j+1)} \big( p-1-k     \big)^{p-1-k -j}     \leqslant    1  .
$$
We may now write
\BEAS
 \frac{  (p-1-k) \cdots (p-k-j+1) }{p \cdots (p-j+1)} & =  & 
 \frac{ (p-1-k)! }{ (p-k-j)!} \frac{ (p-j)!}{p!} = \frac{ (p-1-k)! }{ p! } \frac{ (p-j)!}{ (p-k-j)!} \\
 & = &  \frac{  (p-j) \cdots (p-k-j+1) }{p \cdots (p-k)},
\EEAS
so that we only need to show that
 \BEAS
\frac{136}{15}\sum_{k=0}^{p-1-j}B^{-1-k  } p^{-k-p+j}
{ 2p \choose 2k+2}  
 \frac{  (p-j) \cdots (p-k-j+1) }{p \cdots (p-k)} \big(  p-1-k      \big)^{p-1-k -j}  &  \leqslant  & 1 .
 \EEAS
 We have, by bounding all terms then than $p$ by $p$:
 \BEAS
 & & \frac{136}{15}\sum_{k=0}^{p-1-j} A^{-1-k  } p^{-k-p+j}
{ 2p \choose 2k+2}  
 \frac{  (p-j) \cdots (p-k-j+1) }{p \cdots (p-k)} \big(  p-1-k     \big)^{p-1-k -j}   \\
 & \leqslant & 
 \frac{136}{15}\sum_{k=0}^{p-1-j} A^{-1-k  } p^{-k-p+j}
{ 2p \choose 2k+2}  
 \frac{ p^k  }{p \cdots (p-k)} p^{p-1-k -j}   \\
 & = & 
 \frac{136}{15}\sum_{k=0}^{p-1-j} A^{-1-k  } p^{-k-1}
{ 2p \choose 2k+2}  
 \frac{ 1  }{p \cdots (p-k)}   \\
 & = & 
 \frac{136}{15}\sum_{k=0}^{p-1-j} A^{-1-k  } \frac{p^{-k-1}}{(2k+2)!}
 \frac{ 2p(2p-1)\cdots(2p-2k-1)  }{p \cdots (p-k)}   \\
 & = & 
 \frac{136}{15}\sum_{k=0}^{p-1-j} A^{-1-k  } \frac{p^{-2-1} 2^{2k+2}}{(2k+2)!}
 \frac{ p(p-1/2)\cdots(p-k-1/2)  }{p \cdots (p-k)}   \\
 & \leqslant & 
 \frac{136}{15}\sum_{k=0}^{p-1-j} A^{-1-k  } \frac{  2^{2k+2}}{(2k+2)!}  \\
 & & \mbox{ by associating all } 2k+2  \mbox { terms in ratios which are all less than } 1, \\
& \leqslant & 
 \frac{136}{15}\sum_{k=0}^{+\infty} \frac{  (2/\sqrt{A})^{2k+2}}{(2k+2)!}
 = \frac{136}{15} \big[\cosh  (2/\sqrt{A}) - 1 \big] <1 \mbox{ if } A \leqslant 20
.
 \EEAS
 We thus get the desired result
 $
  \E A_n^{p}
 \leqslant  \big( 3 \|\theta_0 - \theta_\ast \|^2 +  20 n p \gamma^2 R^2  \big)^{p}
 $, and the proposition is proved by induction.

 \subsection{Alternative Proof Using Burkholder-Rosenthal-Pinelis Inequality}
 \label{app:newsto}
 In this section, we present (a slightly modified version of) the proof from \citet{newsto} which is based on Burkholder-Rosenthal-Pinelis inequality \citep[Theorem 4.1]{pinelis}, which we now recall.

 \subsubsection{BRP Inequality}
 Throughout the proof, we use the notation for $X \in \H$ a random vector, and $p$ any {real} number greater than $1$, $\| X\|_p = \big( \E \| X\|^p \big)^{1/p}$. We first recall the Burkholder-Rosenthal-Pinelis (BRP) inequality \citep[Theorem 4.1]{pinelis}. Let $p \in \rb$, $p \geqslant 2$ and $(\mathcal{F}_n)_{n \geqslant 0}$ be a sequence of increasing $\sigma$-fields, and $(X_n)_{n \geqslant 1}$ an adapted sequence of elements of $\H$, such that
$\E \big[ X_n | \F_{n-1} \big] = 0$, and $\| X_n\|_p$ is finite. Then,
\BEA
\label{eq:brp}
\bigg\| \sup_{k\in\{1,\dots,n\}} \bigg\| \sum_{j=1}^k X_j  \bigg \| \bigg\|_p
& \leqslant &  \sqrt{p} \bigg\|
\sum_{k=1}^n \E \big[ \| X_k\|^2 | \F_{k-1} \big]
\bigg\|_{p/2}^{1/2}
+ p \bigg\| \sup_{k\in\{1,\dots,n\}} \| X_k \| \bigg\|_p \\
\nonumber & \leqslant &  \sqrt{p} \bigg\|
\sum_{k=1}^n \E \big[ \| X_k\|^2 | \F_{k-1} \big]
\bigg\|_{p/2}^{1/2}
+ p \bigg\| \sup_{k\in\{1,\dots,n\}} \| X_k \|^2 \bigg\|_{p/2}^{1/2}.
\EEA

\subsubsection{Proof of Proposition~\ref{prop:boundp} (With Slightly Worse Constants)}
 We use BRP's inequality in \eq{brp} to get, for $p \in [2,n/4]$:
\BEAS
\Big\| \sup_{k \in \{0,\dots, n\} } A_k \Big\|_p
& \leqslant & \| \theta_{0 }- \theta_\ast \|^2
+ n \gamma^2   R^2 
+ \sqrt{p} \ \bigg\|
16 \gamma^2 R^2 \sum_{k=1}^n \| \theta_{k-1} - \theta_\ast \|^2
\bigg\|_{p/2}^{1/2} \\
&&  \hspace*{4cm}
+ p \ \Big\| \sup_{k \in \{1,\dots, n\} } 4 \gamma R \| \theta_{k-1} - \theta_\ast\|  \Big\|_p
\\
& \leqslant & \| \theta_{0 }- \theta_\ast \|^2
+ n \gamma^2   R^2 
+ \sqrt{p} \ 4 \gamma R  \sqrt{n}  \Big\|
  \sup_{k \in \{0,\dots, n-1\} } A_k
\Big\|_{p/2}^{1/2} \\
& & \hspace*{5cm}
+ p \ 4 \gamma R \Big\| \sup_{k \in \{0,\dots, n-1\} }  A_{k}^{1/2}  \Big\|_p
\\
& \leqslant & \| \theta_{0 }- \theta_\ast \|^2
+ n \gamma^2   R^2 
+ 4 \gamma R     \Big\|
  \sup_{k \in \{0,\dots, n-1\} } A_k
\Big\|_{p/2}^{1/2} \big( \sqrt{p n } + p \big)
.\EEAS
Thus if $B= \Big\| \sup_{k \in \{0,\dots, n\} } A_k \Big\|_p$, we have (using $p \leqslant n/4$, which implies $\sqrt{p n } + p \leqslant \frac{3}{2}\sqrt{pn}$):
$$
B \leqslant  \| \theta_{0 }- \theta_\ast \|^2
+ n \gamma^2   R^2 
+   6 \gamma R   B^{1/2}  \sqrt{p n }   .$$
By solving this quadratic inequality, we get:
$$
\big( B^{1/2}  - 3 \gamma R   \sqrt{p n }   \big)^2
\leqslant  \| \theta_{0 }- \theta_\ast \|^2
+ n \gamma^2   R^2 
+   9 \gamma^2 R^2 pn ,  $$
which implies
\BEAS
 B^{1/2}  & \leqslant &  3 \gamma R  \sqrt{p n }  
+ \sqrt{  \| \theta_{0 }- \theta_\ast \|^2
+ n \gamma^2   R^2 
+  9 \gamma^2 R^2 pn } \\
B & \leqslant &  2 \times 9 \gamma^2 R^2   {p n }  
+ 2 \times \big(   \| \theta_{0 }- \theta_\ast \|^2
+ n \gamma^2   R^2 
+  9 \gamma^2 R^2 pn \big) \\
& \leqslant & 40 \gamma^2 R^2 pn
+ 2 \| \theta_{0 }- \theta_\ast \|^2.
\EEAS
The previous statement is valid for $p \geqslant 2$ and trivial for $p=1$. From~\myapp{small}, we only need to have the result for $p \leqslant \frac{n}{4}$. Thus the bound is slightly worse (but could be clearly improved with more care, for example,  by using induction on $n$).

 \subsection{Alternative Proof Using Freedman's Inequality}
 \label{app:suggested}
 
 In the previous section, we have used $p$-th order moment martingale inequalities that relate the norm of  a martingale to the norm of its predictable quadratic variation process. Similar results may be obtained for tail bounds through Freedman's inequality \citep[Theorem 1.6]{freedman}. This proof technique was suggested and outlined by Alekh Agarwal (personal communication).
 
 \subsubsection{Freedman's Inequality and Extensions}
 Let $(X_n)$ be a real-valued martingale increment adapted to the increasing sequence of $\sigma$-fields $(\F_n)$, that is, such that $\E (X_n | \mathcal{F}_n) = 0$, that is almost surely bounded, that is, $|X_n| \leqslant R$ almost surely. Let $\Sigma_n = \sum_{k=1}^n \E ( X_k^2 | \F_{k-1})$ the predictable quadratic variation process. Then for any constants $t$ and $\sigma^2$,
 $$
 \P \Big(
 \max_{k \in \{1,\dots,n\}}  \sum_{i=1}^k X_i   \geqslant t, \Sigma_n \leqslant \sigma^2
 \Big) \leqslant 2 \exp \Big( \frac{-t^2}{2( \sigma^2 + Rt / 3)} \Big).
 $$
 When $(X_n)$ are independent random variables, this recovers Bernstein's inequality. From this bound, one may derive the following bound \citep{kakade2009generalization}; with probability $1-4 (\log n) \delta$, we have:
 \BEQ
 \label{eq:freedman} \max_{k \in \{1,\dots,n\}}  \sum_{i=1}^k X_i \leqslant \max \Big\{
 2 \sqrt{ \Sigma_n}, 3R \sqrt{\log\textstyle \frac{1}{\delta}}
 \Big\} \sqrt{\log \textstyle \frac{1}{\delta}}
 \leqslant   
 2 \sqrt{ \Sigma_n} \sqrt{\log \textstyle \frac{1}{\delta}} +   3R \ {\log\textstyle \frac{1}{\delta}}.
 \EEQ
 Note that the result of \citet{kakade2009generalization} considers only $\displaystyle \sum_{i=1}^n X_i $ rather than \linebreak[4]
 $\displaystyle\max_{k \in \{1,\dots,n\}}  \sum_{i=1}^k X_i$, but that the extension of their proof is straightforward.
 
  \subsubsection{Proof of Proposition~\ref{prop:dev} (With Slightly Worse Constants and Scalings)}
We can now apply the inequality in \eq{freedman} to $(M_n)$. We have
$|M_n| \leqslant 4 \gamma R \| \theta_{n-1} - \theta_\ast\| \leqslant 4 \gamma R
 \big( \| \theta_0 - \theta_\ast\| + n \gamma R \big)
$ almost surely. Moreover, $\E (M_n^2 | \F_{n-1} ) \leqslant 16 \gamma^2 R^2 \| \theta_{n-1} - \theta_\ast\|^2 \leqslant 16 \gamma^2 R^2 A_{n-1}$.

This leads to with probability greater than $1-4 (\log n) \delta$,
\BEAS
 \max_{k \in \{1,\dots,n\}}  A_k\!\!\!
&  \leqslant   &
 \| \theta_0 - \theta_\ast\|^2 + n \gamma^2 R^2 + 
 8 \gamma R \sqrt{\sum_{k=1}^{n-1} A_k}  \sqrt{\log \textstyle \frac{1}{\delta}} +   12 \gamma R
 \big( \| \theta_0 - \theta_\ast\| + n \gamma R \big)
  \ {\log\textstyle \frac{1}{\delta}} \\
  & \leqslant &  \| \theta_0 - \theta_\ast\|^2 + n \gamma^2 R^2 + 
 8 \gamma R \sqrt{n}  \max_{k \in \{1,\dots,n\}}  \sqrt{A_k}  \sqrt{\log \textstyle \frac{1}{\delta}} \\
 & & \hspace*{7cm} +   12 \gamma R
 \big( \| \theta_0 - \theta_\ast\| + n \gamma R \big)
  \ {\log\textstyle \frac{1}{\delta}}.
 \EEAS
 We may now solve the quadratic inequality in $\max_{k \in \{1,\dots,n\}}  \sqrt{A_k} $. This leads to
 \BEAS
 & & \Big( \max_{k \in \{1,\dots,n\}}  \sqrt{A_k}
 - 4 \gamma R \sqrt{n}   \sqrt{\log \textstyle \frac{1}{\delta}}
 \Big)^2 \\
&  \leqslant   &
 \| \theta_0 - \theta_\ast\|^2 + n \gamma^2 R^2   + 12 \gamma R
 \big( \| \theta_0 - \theta_\ast\| + n \gamma R \big)
  \ {\log\textstyle \frac{1}{\delta}}
  + 16 \gamma^2 R^2 n {\log \textstyle \frac{1}{\delta}} \\
 &  =   &
 \| \theta_0 - \theta_\ast\|^2 + n \gamma^2 R^2   + 
 \big( 12 \gamma R \| \theta_0 - \theta_\ast\| + 28  n \gamma^2 R^2 \big)
  \ {\log\textstyle \frac{1}{\delta}}.
 \EEAS
 Then
 \BEAS
   & & \max_{k \in \{1,\dots,n\}}  \sqrt{A_k} \\
  & \leqslant & 4 \gamma R \sqrt{n}   \sqrt{\log   \frac{1}{\delta}}
  + 
   \| \theta_0 - \theta_\ast\|  + \sqrt{n} \gamma  R    + 
 \sqrt{  12 \gamma R \| \theta_0 - \theta_\ast\| + 28  n \gamma^2 R^2 }
  \ \sqrt{\log  \frac{1}{\delta}} 
  \EEAS
  and
  \BEAS
& &     \max_{k \in \{1,\dots,n\}}    {A_k} \\
 & \leqslant & 64 \gamma^2 R^2 n    {\log   \frac{1}{\delta}}
  + 
  4  \| \theta_0 - \theta_\ast\|^2  + 4  {n} \gamma^2  R^2    + 
 4 \big(  12 \gamma R \| \theta_0 - \theta_\ast\| + 28  n \gamma^2 R^2 \big)
  \ {\log  \frac{1}{\delta}} \\
  & \leqslant & 
   4  \| \theta_0 - \theta_\ast\|^2  + 4  {n} \gamma^2  R^2
   + \Big(
    64 \gamma^2 R^2 n  
    + 48 \gamma R \| \theta_0 - \theta_\ast\| + 112  n \gamma^2 R^2 
   \Big)  \log  \frac{1}{\delta}
\\
  & \leqslant & 
   4  \| \theta_0 - \theta_\ast\|^2  + 4  {n} \gamma^2  R^2
   + \Big(
    176 \gamma^2 R^2 n  
    + 48 \gamma R \| \theta_0 - \theta_\ast\|  
   \Big)  \log  \frac{1}{\delta}.
 \EEAS
We thus recover a tail bound which is very similar to the one obtained in Proposition~\ref{prop:dev}, with the following differences: the additional term
$48 \gamma R \| \theta_0 - \theta_\ast\|  $ is unimportant because $\gamma = O(N^{-1/2})$; however, because the extension of Freedman's inequality is satisfied with probability $1-4 (\log n) \delta$, this proof technique loses a logarithmic factor.

\section{Proof of Proposition~\ref{prop:grad}}

The proof is organized in two parts:  we first show a bound on the averaged gradient $\frac{1}{n} \sum_{k=1}^n f'(\theta_{k-1})$, then relate it to  the gradient at the averaged iterate, that is, \linebreak[4] $ f' \Big(\frac{1}{n} \sum_{k=1}^n \theta_{k-1}   \Big) $, using self-concordance.

\subsection{Bound on $\frac{1}{n} \sum_{k=1}^n f'(\theta_{k-1})$}
We have, following \citet{polyak1992acceleration} and \citet{gradsto}:
$$
f_n'(\theta_{n-1}) = \frac{1}{\gamma}( \theta_{n-1} - \theta_{n}),
$$
which implies, by summing over all integers between $1$ and $n$:
\BEAS
\frac{1}{n} \sum_{k=1}^n f'(\theta_{k-1})
& = & \frac{1}{n} \sum_{k=1}^n \big[ f'(\theta_{k-1}) - f'_k(\theta_{k-1}) \big]
+ \frac{1}{\gamma n} ( \theta_0 - \theta_\ast) 
+ \frac{1}{\gamma n} ( \theta_\ast - \theta_n) .
\EEAS

We denote  $X_k = \frac{1}{n} \big[  f'(\theta_{k-1}) - f'_k(\theta_{k-1}) \big] \in \H$. We have: $\|X_k\| \leqslant \frac{2R}{n}$ almost surely and $\E (X_k | \F_{k-1}) = 0$, with
$ \big( \sum_{k=1}^n \E (\|X_k\|^2|\F_{k-1}) \big)^{1/2} \leqslant \frac{2R}{\sqrt{n}}$.
We may thus apply the Burkholder-Rosenthal-Pinelis inequality \citep[Theorem 4.1]{pinelis}, and get:
\BEAS
\bigg[ \E \bigg\|
\frac{1}{n} \sum_{k=1}^n \big[ f'(\theta_{k-1}) - f'_k(\theta_{k-1}) \big] \bigg\|^{2p}
  \bigg]^{1/2p}
& \leqslant &   2p \frac{2R}{n} + \sqrt{2p} \frac{2R}{n^{1/2}}  .
\EEAS
This leads to, using Proposition~\ref{prop:boundp} and Minkowski's inequality:
\BEA
 \nonumber  & & \bigg[ \E  \bigg\|
\frac{1}{n} \sum_{k=1}^n f'(\theta_{k-1}) \bigg\|^{2p}    \bigg]^{1/2p} \\
\nonumber & \leqslant &  
\bigg[ \E \bigg\|
\frac{1}{n} \sum_{k=1}^n \big[ f'(\theta_{k-1}) - f'_k(\theta_{k-1}) \big] \bigg\|^{2p}
  \bigg]^{1/2p}  
   + \frac{1}{\gamma n} \| \theta_0 - \theta_\ast \|
+ \frac{1}{\gamma n} \big[ \E \| \theta_\ast - \theta_n \|^{2p} \big]^{1/2p}
\\
\nonumber  & \leqslant & 
 2p \frac{2R}{n} + \sqrt{2p} \frac{2R}{n^{1/2}} 
 + \frac{1}{\gamma n} \| \theta_0 - \theta_\ast\| +
\big[ \frac{1}{\gamma n} \sqrt{  3\|\theta_0 - \theta_\ast \|^2 +  20 n p \gamma^2 R^2}  \big]\\
\nonumber  & \leqslant & 
  2p \frac{2R}{n} + \sqrt{2p} \frac{2R}{n^{1/2}} 
 + \frac{1}{\gamma n} \| \theta_0 - \theta_\ast\| +
 \big[ \frac{\sqrt{3}}{\gamma n}  \|\theta_0 - \theta_\ast \|  +  \frac{1}{\gamma n} \sqrt{20 n p} \gamma R \big] \\
\nonumber  & \leqslant & 
   \frac{4pR}{n} + \sqrt{2p} \frac{2R}{n^{1/2}} 
 + \frac{2}{\gamma n} \| \theta_0 - \theta_\ast\|  +  \frac{1}{\gamma n} \sqrt{20 n p} \gamma R  \\
\nonumber  & \leqslant & 
 \frac{4pR}{n} +  \sqrt{p} \frac{R}{\sqrt{n}}
  \big[
   2\sqrt{2} + \sqrt{20}
  \big] 
 + \frac{1+ \sqrt{3} }{\gamma n} \| \theta_0 - \theta_\ast\|
 \\
 \label{eq:BC}
 & \leqslant & 
  \frac{4pR}{n} + 8 \sqrt{p} \frac{R}{\sqrt{n}}
 + \frac{3}{\gamma n} \| \theta_0 - \theta_\ast\|.
 \EEA
  
  \subsection{Using Self-Concordance}
Using the self-concordance property of Lemma~\ref{lemma:selfcontrol} several times, we obtain:
\BEAS
\!\!\!  &\!\!\!  \!\!\!  & \bigg\| \frac{1}{n} \sum_{k=1}^n f'(\theta_{k-1}) 
- f' \bigg(\frac{1}{n} \sum_{k=1}^n \theta_{k-1}   \bigg) \bigg\| \\
\!\!\! & \!\!\! \!\!\! \!\!\! =  \!\!\! & 
\bigg\| \frac{1}{n} \sum_{k=1}^n \big[ f'(\theta_{k-1}) - f'(\theta_\ast) 
- f''(\theta_\ast) ( \theta_{k-1} - \theta_\ast) \big] \\
& & \hspace*{5cm}
- f' \bigg(\frac{1}{n} \sum_{k=1}^n \theta_{k-1}   \bigg)   + f'(\theta_\ast) 
+ f''(\theta_\ast) \bigg(\frac{1}{n} \sum_{k=1}^n \theta_{k-1} - \theta_\ast \bigg) \bigg\| \\
\!\!\! & \!\!\! \!\!\!  \!\!\! \leqslant  \!\!\!  & \frac{R}{n} \sum_{k=1}^n \big[ f(\theta_{k-1}) - f(\theta_\ast) 
- \langle f'(\theta_\ast) , \theta_{k-1} - \theta_\ast \rangle \big]
\\
& & \hspace*{5cm}
+ R \bigg[
f \bigg(\frac{1}{n} \sum_{k=1}^n \theta_{k-1}   \bigg)   - f(\theta_\ast) 
+ \bigg\langle f'(\theta_\ast) , \frac{1}{n} \sum_{k=1}^n \theta_{k-1} - \theta_\ast \bigg\rangle 
\bigg]
\\
\!\!\! &\!\!\! \!\!\!\!\!\!    \leqslant  \!\!\! & 2R \bigg( \frac{1}{n} \sum_{k=1}^n  f(\theta_{k-1}) - f(\theta_\ast) \bigg)
\mbox{ using the convexity of $f$.}
\EEAS

This leads to, using Proposition~\ref{prop:boundp}:
\BEA
\nonumber  & &\bigg( \E \bigg\| \frac{1}{n} \sum_{k=1}^n f'(\theta_{k-1}) 
- f' \bigg(\frac{1}{n} \sum_{k=1}^n \theta_{k-1}   \bigg) \bigg\|^{2p} \bigg)^{1/2p} \\
\label{eq:BV} & \leqslant & 2R \bigg( \E \bigg[\frac{1}{n} \sum_{k=1}^n  f(\theta_{k-1}) - f(\theta_\ast) \bigg]^{2p} \bigg)^{1/2p}   \leqslant   \frac{2R}{2\gamma  n}
\bigg(
3\|\theta_0 - \theta_\ast\|^2 + 40 n  p \gamma^2 R^2
\bigg).
\EEA
Summing \eq{BC} and \eq{BV}   leads to the desired result.

\section{Results for Small $p$}
\label{app:F}
In Proposition~\ref{prop:boundp}, we may replace the bound $3\| \theta_0 - \theta_\ast\|^2 + 20 n  p\gamma^2 R^2$ with a bound with smaller constants for $p=1,2,3$ (to be used in proofs of results in \mysec{self}).
This is done using the same proof principle but finer derivations, as follows. 
We denote $\gamma^2 R^2 = b$ and $\|\theta-\theta_\ast\|^2 = a$, and consider the following inequalities which we have considered in the proof of Proposition~\ref{prop:boundp}:
\BEAS
A_n^p & \leqslant & ( A_{n-1} + b + M_{n} )^p \\
M_{n} & \leqslant & 4 b^{1/2} A_{n-1}^{1/2} \mbox{ and } \E (M_{n} | \F_{n-1}) = 0,\\
A_0 & = & a.
\EEAS
We simply take expansions of the $p$-th power above, and sum for all first integers.
We have:
\BEAS
\E A_n \!\!\! & \leqslant & \E A_{n-1} + b \leqslant a + n b, \\
\E A_n^2 \!\!\! & \leqslant & \E( A_{n-1}^2 + b^2 + 2 b A_{n-1} + M_n^2 ) \leqslant \E A_{n-1}^2 + 2 \E A_{n-1}b + b^2 + 16 b \E A_{n-1} \\
& \leqslant &  a^2 + 18b \bigg[ \sum_{k=0}^{n-1} a + k b \bigg] + b^2 n  
\leqslant  a^2 + 18b [ na + \frac{n^2}{2} b ] + b^2 n   \\
& & \mbox{ using the result about } \E A_{n-1},
\\
& = & a^2 + 18 b na +  b^2 ( n + 9 n^2) \\
& \leqslant & (a + 9nb)^2.
\EEAS
We may now pursue for the third order moments:
\BEAS
\E A_n^3 \!\!\! & \leqslant & \E ( A_{n-1} +  b )^3 + 3 \E (A_{n-1} + b)^2 M_n^2  + 3 \E (A_{n-1} + b)^3 M_n   +    \E M_{n-1}^3
\\
& \leqslant & \E ( A_{n-1} +  b )^3 + 3 \E (A_{n-1} + b)^2 16 b A_{n-1} + 0 + 64 b^{3/2} \E A_{n-1}^{3/2}
\\
& \leqslant & ( \E A_{n-1}^3 + 3\E A_{n-1}^2b+3 \E A_{n-1} b^2 + b^3 ) + 3\E ( A_{n-1} + b) 16 b A_{n-1} + 64 b^{3/2} \E A_{n-1}^{3/2}
\\
 & = & ( \E A_{n-1}^3 + 3\E A_{n-1}^2b+3 \E A_{n-1} b^2 + b^3 ) + 3\E ( A_{n-1} + b) 16 b A_{n-1}\\
 & & \hspace*{8cm} + 32 b\E A_{n-1} [ 2 b^{1/2} A_{n-1}^{1/2}].
 \EEAS
 By expanding, we get
\BEAS
\E A_n^3 \!\!\! & \leqslant & ( \E A_{n-1}^3 + 3\E A_{n-1}^2b+3 \E A_{n-1} b^2 + b^3 ) + 3 \E ( A_{n-1} + b) 16b A_{n-1} 
 \\
 & & \hspace*{8cm} + 32 \E  bA_{n-1} [ \frac{A_{n-1}}{4} + 4 b]
 \\
  & = &  E A_{n-1}^3 + \E A_{n-1}^2b [ 3 + 48 + 8] +  \E A_{n-1} b^2 [ 3+ 48 + 128] + b^3
 \\
  & = &  \E A_{n-1}^3 +  59 \E A_{n-1}^2b  +  179 \E A_{n-1} b^2   + b^3
\\
  & \leqslant &  a^3 +  59  b  \bigg[ \sum_{k=1}^{n-1} a^2 + 18 b k a +  b^2 ( k + 9 k^2) \bigg]+  179   b^2 
   \bigg[ \sum_{k=1}^{n-1} a + kb \bigg]+ n b^3\\
  & \leqslant &  a^3 +  59  b  [ na^2 + 9 bn^2 a + b^2( n^2/2 + 3 n^3) ]+  179   b^2 
  [ n a + bn^2/2  ]  + n b^3\\
   & = &  a^3 +  59 n b a^2 
   +   b^2 a [ 59 \cdot 9  n^2 + 179 n  ]
   + b^3 [ 59/2 \cdot n^2 + 3 \cdot 59 n^3 + 179/2 \cdot n^2 + n ]
\\
   & = &  a^3 +  59 n b a^2 
   +   b^2 a [ 531 n^2 + 179 n ]
   + b^3 [ 119 n^2 + 177 n^3  + n ] \\
   & \leqslant & 
  ( a + 20 nb)^3.
  \EEAS

      We then obtain:
      \BEAS
  \E \bigg[
2\gamma n  \big[  f ( \bar{\theta}_n) -   f(\theta^\ast)  \big]+\| \theta_n - \theta_\ast \|^{2} \bigg]^2
&  \leqslant  & \big( \|\theta_0 - \theta_\ast \|^2 +  9 n  \gamma^2 R^2  \big)^{2}\\
  \E \bigg[
2\gamma n  \big[  f ( \bar{\theta}_n) -   f(\theta^\ast)  \big]+\| \theta_n - \theta_\ast \|^{2} \bigg]^3
&  \leqslant  & \big( \|\theta_0 - \theta_\ast \|^2 +  20 n  \gamma^2 R^2  \big)^{3}.
\EEAS

\section{Proof of Proposition~\ref{prop:boundself}}
\label{sec:boundself}

The proof follows from applying self-concordance properties (Lemma~\ref{prop:selfc})   to $\bar{\theta}_n$. We thus need to provide a control on the probability that $\| f'(\bar{\theta}_n )\|  \geqslant \frac{3\mu}{4R}$.

\subsection{Tail Bound for $\| f'(\bar{\theta}_n )\|$}
\label{app:tail}
  We  derive a large deviation bound, as a consequence of the bound on all moments of $\| f'(\bar{\theta}_n )\|$  (Proposition~\ref{prop:grad}) and Lemma~\ref{lemma:conc2}, that allows to go from moments to tail bounds:
$$  \P \bigg( \big\| f'(\bar{\theta}_n )\| \geqslant 
 \frac{2R}{\sqrt{n}}
\bigg[
 {10}{\sqrt{t}} + 40  R^2 \gamma t \sqrt{n}
 +
\frac{3}{ \gamma  \sqrt{n}}
\|\theta_0 - \theta_\ast\|^2  + \frac{3}{\gamma R \sqrt{n}} \| \theta_0 - \theta_\ast\|
\bigg]  \bigg) \leqslant  4 \exp( -  t ).$$
In order to derive the bound above, we need to assume that $p \leqslant n/4$ (so that $4p/n \leqslant 2 \sqrt{p} / \sqrt{n})$, and thus, when applying Lemma~\ref{lemma:conc2}, the bound above is valid as long as $t \leqslant n/4$. It is however valid for all $t$, because the gradients are bounded by $R$, and for $t > n$, we have $\frac{2R}{\sqrt{n}}
 {10}{\sqrt{t}}  \geqslant R$, and the inequality is satisfied with zero probability.

\subsection{Bounding the Function Values}
\label{app:boundF}
From  Lemma~\ref{prop:selfc}, if
$\| f'(\bar{\theta}_n )\|  \geqslant \frac{3\mu}{4R}$, then $f(\bar{\theta}_n) - f(\theta_\ast)\leqslant 2 \frac{\| f'(\bar{\theta}_n )\| ^2}{\mu}$. This will allow us to derive a tail bound for $f(\bar{\theta}_n) - f(\theta_\ast)$, for sufficiently small deviations. For larger deviations, we will use the tail bound which does not use strong convexity (Proposition~\ref{prop:dev}).

 We consider  the event
 \BEAS A_t &  = &  \bigg\{\big\| f'(\bar{\theta}_n) \| \leqslant 
 \frac{2R}{\sqrt{n}}
\bigg[
 {10}{\sqrt{t}} + 40  R^2 \gamma t \sqrt{n}
 +
\frac{3}{ \gamma  \sqrt{n}}
\|\theta_0 - \theta_\ast\|^2  + \frac{3}{\gamma R \sqrt{n}} \| \theta_0 - \theta_\ast\|
\bigg]  \bigg\}.
   \EEAS
  We make the following two assumptions regarding $\gamma$ and $t$:
  \BEA
  \label{eq:AAA}
    {10}{\sqrt{t}} + 40   R^2 \gamma t \sqrt{n}  & \leqslant & \frac{2}{3}  \frac{3\mu}{4R }   \frac{\sqrt{n}}{2R}  =     \frac{\mu \sqrt{n}}{4R^2} \\
 \nonumber \mbox{ and } \frac{3}{ \gamma  \sqrt{n}}
\|\theta_0 - \theta_\ast\|^2   +
   \frac{3}{\gamma R \sqrt{n}} \| \theta_0 - \theta_\ast\|
  & \leqslant & \frac{1}{3} \frac{3\mu}{4R }   \frac{\sqrt{n}}{2R} =       \frac{\mu \sqrt{n}}{8R^2},
   \EEA
   so that the upper-bound on $\| f'(\bar{\theta}_n) \|$ in the definition of $A_t$ is less than $  \frac{3\mu}{4R }$ (so that we can apply  Lemma~\ref{prop:selfc}). We thus have: 
   \BEAS A_t & \subset & \bigg\{
   f(\bar{\theta}_n) - f(\theta_\ast)
    \leqslant  
 \frac{8R^2}{\mu n}
\bigg[
 {10}{\sqrt{t}} + 40  R^2 \gamma t \sqrt{n}
 +
\frac{3}{ \gamma  \sqrt{n}}
\|\theta_0 - \theta_\ast\|^2  + \frac{2}{\gamma R \sqrt{n}} \| \theta_0 - \theta_\ast\|
\bigg]^2 \bigg\} 
   \\
   & \subset & \bigg\{
   f(\bar{\theta}_n) - f(\theta_\ast)
    \leqslant 
 \frac{8R^2}{\mu n}
\bigg[
 {10}{\sqrt{t}} + 20 \square t
 +  \triangle \bigg]^2 \bigg\} 
   ,
   \EEAS
   with $\displaystyle \square = 2 \gamma R^2 \sqrt{n}$ and $\displaystyle \triangle = \frac{3}{ \gamma  \sqrt{n}}
\|\theta_0 - \theta_\ast\|^2  + \frac{3}{\gamma R \sqrt{n}} \| \theta_0 - \theta_\ast\|$.

   This implies that  for all $t \geqslant 0$,
such that   ${10}{\sqrt{t}} + 20 \square t    \leqslant         \frac{\mu \sqrt{n}}{4R^2}  $, that is, our assumption in \eq{AAA}, we may apply the tail bound from Appendix~\ref{app:tail} to get:
   \BEQ
   \label{eq:tailbound1}
   \P \bigg(
     f(\bar{\theta}_n) - f(\theta_\ast) \geqslant \frac{8R^2}{\mu n}
\bigg[
 {10}{\sqrt{t}} + 20 \square t
 +  \triangle \bigg]^2  
 \bigg) \leqslant 4 e^{-t}.
 \EEQ
 Moreover, we have for all $v \geqslant 0$ (from Proposition~\ref{prop:dev}):
   \BEQ
   \label{eq:tailbound2}
\P\bigg(
  f ( \bar{\theta}_n) -   f(\theta_\ast)  \geqslant  30 \gamma R^2 v + \frac{ 3 \| \theta_0 - \theta_\ast\|^2}{\gamma n} \bigg)
  \leqslant 2 \exp (-v) .
  \EEQ
  We may now use the last two inequalities to bound the expectation   $\E [ f ( \bar{\theta}_n) -   f(\theta_\ast) ]$.
  
  We   first express the expectation as an integral of the tail bound and split it into three parts:
\BEA
\nonumber \E \big[ f(\bar{\theta}_n) - f(\theta_\ast) \big]
& = & \int_{0}^{+\infty}
\P \big[ 
f(\bar{\theta}_n) - f(\theta_\ast) \geqslant u
\big] du\\
\label{eq:proofAA} & = & \int_{0}^{\triangle^2 \frac{8R^2}{ {\mu n}} }
\P \big[ 
f(\bar{\theta}_n) - f(\theta_\ast) \geqslant u
\big] du \\
\nonumber & & 
+
\int_{\triangle^2 \frac{8R^2}{ {\mu n}} }^{\frac{8R^2}{ {\mu n}}\big( {
 \frac{\mu \sqrt{n}}{4R^2} +  \triangle } \big)^2}
\P \big[ 
f(\bar{\theta}_n) - f(\theta_\ast) \geqslant u
\big] du\\
\nonumber &&  +
\int^{+\infty}_{\frac{8R^2}{ {\mu n}}\big( {
 \frac{\mu \sqrt{n}}{4R^2} +  \triangle } \big)^2}
\P \big[ 
f(\bar{\theta}_n) - f(\theta_\ast) \geqslant u
\big] du.
\EEA
We may now bound the three terms separately. For the first integral, we bound the probability by one to get
$\displaystyle \int_{0}^{\triangle^2 \frac{8R^2}{ {\mu n}} }
\P \big[ 
f(\bar{\theta}_n) - f(\theta_\ast) \geqslant u
\big] du \leqslant  \triangle^2 \frac{8 R^2}{n\mu}$.

For the third term in \eq{proofAA}, we use the tail bound in \eq{tailbound2} to get
\BEAS
& & \int^{+\infty}_{\frac{8R^2}{ {\mu n}}\big( {
 \frac{\mu \sqrt{n}}{4R^2} +  \triangle } \big)^2}
\P \big[ 
f(\bar{\theta}_n) - f(\theta_\ast) \geqslant u
\big] du \\
& = & 
\int^{+\infty}_{\frac{8R^2}{ {\mu n}}\big( {
 \frac{\mu \sqrt{n}}{4R^2} +  \triangle } \big)^2 - \frac{3}{\gamma n} \| \theta_0 - \theta_\ast\|^2}
\P \bigg[ 
f(\bar{\theta}_n) - f(\theta_\ast) \geqslant u   + \frac{3}{\gamma n} \| \theta_0 - \theta_\ast\|^2
\bigg] du
\\
& \leqslant & 2
\int^{+\infty}_{\frac{8R^2}{ {\mu n}}\big( {
 \frac{\mu \sqrt{n}}{4R^2} +  \triangle } \big)^2 - \frac{3}{\gamma n} \| \theta_0 - \theta_\ast\|^2}
 \exp\big(
 -\frac{ u}{30 \gamma R^2 }
 \big)du.
\EEAS
We may apply \eq{tailbound2} because 
$$
\frac{8R^2}{ {\mu n}}\big( {
 \frac{\mu \sqrt{n}}{4R^2} +  \triangle } \big)^2 - \frac{3}{\gamma n} \| \theta_0 - \theta_\ast\|^2
 \geqslant \frac{8R^2}{ {\mu n}}\big( {
 \frac{\mu \sqrt{n}}{4R^2}\!  +  \!\triangle } \big)^2  - \frac{\mu}{8R^2} 
 \geqslant  \frac{8R^2}{ {\mu n}}\big( {
 \frac{\mu \sqrt{n}}{4R^2}  } \big)^2  - \frac{\mu}{8R^2}  = \frac{3 \mu}{8R^2}  \geqslant 0. $$
We can now compute the bound explicitly to get
\BEAS
& & \int^{+\infty}_{\frac{8R^2}{ {\mu n}}\big( {
 \frac{\mu \sqrt{n}}{4R^2} +  \triangle } \big)^2}
\P \big[ 
f(\bar{\theta}_n) - f(\theta_\ast) \geqslant u
\big] du \\
& \!\!\!\!\!\! \leqslant \!\!\!\!\!\!  & 
60 \gamma R^2  
 \exp\bigg( 
  \frac{ -1}{30 \gamma R^2 }
 \bigg[
 \frac{8R^2}{ {\mu n}}\big( {
 \frac{\mu \sqrt{n}}{4R^2} +  \triangle } \big)^2 \!-\! \frac{3}{\gamma n} \| \theta_0 - \theta_\ast\|^2
 \bigg]
 \bigg)  \leqslant 
 60 \gamma R^2  
 \exp\bigg( 
 \frac{- 1}{30 \gamma R^2 }
 \frac{3 \mu}{8R^2} 
 \bigg) 
 \\
 & \!\!\!\!\!\!  \leqslant \!\!\!\!\!\!  &  60 \gamma R^2  
 \exp\bigg(
 -\frac{ \mu}{80 \gamma R^4  }
 \bigg) 
 \leqslant  60 \gamma R^2  
 \frac{80 \gamma R^4  }{2 \mu}  \mbox{ using } e^{-\alpha} \leqslant \frac{1}{2\alpha} \mbox{ for all } \alpha > 0
\\
& \!\!\!\!\!\! = \!\!\!\!\!\!  &  \frac{2400 \gamma^2 R^6}{ \mu}.
\EEAS

We now consider the second term in \eq{proofAA} for which we will use \eq{tailbound1}. We consider the change of variable
$u=\frac{8R^2}{\mu n}
\Big[
 {10}{\sqrt{t}} + 20 \square t
 +  \triangle \Big]^2$, for which $u \in$ \linebreak[4] $\Big[
 \triangle^2 \frac{8R^2}{\mu n },  \frac{8R^2}{ {\mu n}}\big( {
 \frac{\mu \sqrt{n}}{4R^2} +  \triangle } \big)^2
 \Big]$ implies $t \in [ 0 , +\infty)$. This implies that
\BEAS
  & & \int_{\triangle^2 \frac{8R^2}{ {\mu n}} }^{\frac{8R^2}{ {\mu n}}\big( {
 \frac{\mu \sqrt{n}}{4R^2} +  \triangle } \big)^2}
\P \big[ 
f(\bar{\theta}_n) - f(\theta_\ast) \geqslant u
\big] du   \\
& \leqslant & \int_{0}^\infty 4e^{-t} d \bigg(
 \frac{8R^2}{\mu n}
\bigg[
 {10}{\sqrt{t}} + 20 \square t
 +  \triangle \bigg]^2 
\bigg)  \\
& = &  \frac{32 R^2}{\mu n}  \int_{0}^\infty e^{-t}   \bigg(
 100 + 400 \square^2 2t   + 400  \square \frac{3}{2} t^{1/2} + 20 \triangle \frac{1}{2} t^{-1/2} + 40   \triangle \square \bigg) dt
 \\
 & = &   \frac{32 R^2}{\mu n}   \bigg(
 100 \Gamma(1) + 400 \square^2 2  \Gamma(2)    + 400  \square \frac{3}{2}  \Gamma(3/2)  + 20 \triangle \frac{1}{2}\Gamma(1/2) + 40   \triangle \square \Gamma(1) \bigg)  \\
 & & \mbox{ with }  \Gamma \mbox{ denoting the Gamma function,} \\
 & = & \frac{32 R^2}{\mu n}   \bigg(
 100   + 400 \square^2 2       + 400  \square \frac{3}{2}  \frac{1}{2} \sqrt{\pi} + 20 \triangle \frac{1}{2}\sqrt{\pi} + 40   \triangle \square \bigg) .
\EEAS

We may now combine the three bounds to get, from \eq{proofAA},
\BEAS
\E \big[ f(\bar{\theta}_n) \!-\! f(\theta_\ast) \big]
& \!\!\!\! \!\!\!\!\leqslant \!\!\!\!\!\! \!\!&
 \triangle^2 \frac{8 R^2}{n\mu}  +  \frac{2400 \gamma^2 R^6 }{\mu }\\
 & & + 
\frac{32 R^2}{\mu n}   \bigg(
 100   + 400 \square^2 2       + 400  \square \frac{3}{2}  \frac{1}{2} \sqrt{\pi} + 20 \triangle \frac{1}{2}\sqrt{\pi} + 40   \triangle \square \bigg) 
 \\
 & \!\!\!\! \!\!\!\!\leqslant \!\!\!\!\!\!\!\! & \frac{32 R^2}{n \mu} \bigg[ \frac{\triangle^2}{4} \! +\! 
  {75 \gamma^2 R^4 n  } \!+\!
  100   \!+ \!800 \square^2       \!  +\! 300  \square  \sqrt{\pi} \!+\! 10 \triangle  \sqrt{\pi} \!+\! 40   \triangle \square
 \bigg].
\EEAS

For $\gamma = \frac{1}{2 R^2 \sqrt{N}}$,
 with $\alpha = R \| \theta_0 - \theta_\ast\|$, $\square = 1$ and  $\triangle = 6 \alpha^2 + 6 \alpha$, we obtain 
\BEAS
\E \big[ f(\bar{\theta}_N) \!-\! f(\theta_\ast) \big]
& \!\!\!\!\leqslant \!\!\!\! &    \frac{32 R^2}{N\mu}
\bigg[
\frac{1}{4}\triangle^2 + 1451 + 58 \Delta \bigg] \\
& \!\!\!\!\leqslant \!\!\!\!&    \frac{32 R^2}{N\mu}
\bigg[
9 \alpha^4 + 18 \alpha^3 + 9 \alpha^2 + 1451 + 348 \alpha^2 + 348 \alpha \bigg] \\
& \!\!\!\!\leqslant \!\!\!\!& 
\frac{  R^2}{N\mu} \big(
625 \alpha^4 + 7500 \alpha^3 + 33750 \alpha^2 + 67500 \alpha + 50625
\big)
\!=\! \frac{  R^2}{N\mu}
\big(5 \alpha + 15 \big)^4
.\EEAS

Note that the previous bound is only valid if   $\frac{3}{ \gamma  \sqrt{n}}
\|\theta_0 - \theta_\ast\|^2   +
   \frac{3}{\gamma R \sqrt{n}} \| \theta_0 - \theta_\ast\|
 \leqslant      \frac{\mu \sqrt{n}}{8R^2}$, that is, under the condition $ 6 R^2 \| \theta_0 - \theta_\ast\|^2 +6 R \| \theta_0 - \theta_\ast\| \leqslant  \frac{\mu \sqrt{N}}{8R^2}$. If the condition  is not satisfied, then the bound is still valid because of Lemma~\ref{prop:old}. We thus obtain 
 the desired result.

\subsection{Bound on Iterates}
Following the same principle as for function values in Appendix~\ref{app:boundF}, we consider the same event~$A_t$. With the same condition on $\gamma$ and $t$, we have:
 $$
  A_t \subset   \bigg\{
   \| \bar{\theta}_n - \theta_\ast \|^2
    \leqslant 
 \frac{16 R^2}{\mu^2 n}
\bigg[
 {10}{\sqrt{t}} + 20 \square t
 +  \triangle \bigg]^2 \bigg\} ,
$$
which leads to the tail bound:
$$
   \P \bigg(
        \| \bar{\theta}_n - \theta_\ast \|^2
  \geqslant \frac{16 R^2}{\mu^2 n}
\bigg[
 {10}{\sqrt{t}} + 20 \square t
 +  \triangle \bigg]^2  
 \bigg) \leqslant 4 e^{-t}.
 $$
 We may now split the expectation in three integrals:
 \BEA
  \E \| \bar{\theta}_n - \theta_\ast \|^2 
\label{eq:proofBB}
& = & \int_{0}^{\frac{16R^2}{ {\mu^2 n}} \triangle  ^2}
\P \big[ 
\| \bar{\theta}_n - \theta_\ast \|^2  \geqslant u
\big] du \\
\nonumber&&  +  
 \int_{\frac{16R^2}{ {\mu^2 n}} \triangle  ^2}^{\frac{16R^2}{ {\mu^2 n}}\big( {
 \frac{\mu \sqrt{n}}{4R^2} +  \triangle } \big)^2} 
\P \big[ 
\| \bar{\theta}_n - \theta_\ast \|^2  \geqslant u
\big] du \\
\nonumber&& + \int_{\frac{16R^2}{ {\mu^2 n}}\big( {
 \frac{\mu \sqrt{n}}{4R^2} +  \triangle } \big)^2}^\infty 
\P \big[ 
\| \bar{\theta}_n - \theta_\ast \|^2  \geqslant u
\big] du.
\EEA
The first term in \eq{proofBB} is simply bounded by bounding the tail bound by one (like in the previous section):
$\displaystyle 
 \int_{0}^{\frac{16R^2}{ {\mu^2 n}} \triangle  ^2}
\P \big[ 
\| \bar{\theta}_n - \theta_\ast \|^2  \geqslant u
\big] du
\leqslant \frac{16R^2}{ {\mu^2 n}} \triangle  ^2 $. The last integral in \eq{proofBB} may be bounded as follows:
\BEAS
& & \int_{\frac{16R^2}{ {\mu^2 n}}\big( {
 \frac{\mu \sqrt{n}}{4R^2} +  \triangle } \big)^2}^\infty 
\P \big[ 
\| \bar{\theta}_n - \theta_\ast \|^2  \geqslant u
\big] du
\\
& = & \E \bigg[
1_{\| \bar{\theta}_n - \theta_\ast \|^2  \geqslant \frac{16R^2} {\mu^2 n}\big( {
 \frac{\mu \sqrt{n}}{4R^2} +  \triangle } \big)^2} \| \bar{\theta}_n - \theta_\ast\|^2
\bigg] \\
& \leqslant &  \P \bigg[ \| \bar{\theta}_n - \theta_\ast \|^2  \geqslant \frac{16R^2} {\mu^2 n}\big( {
 \frac{\mu \sqrt{n}}{4R^2} +  \triangle } \big)^2\bigg]^{1/2} \bigg[\E \big( \| \bar{\theta}_n - \theta_\ast\|^4 \big)
\bigg]^{1/2}
\\
& & \mbox{ using Cauchy-Schwarz inequality, } \\
& \leqslant & \P \bigg[ 
\| \bar{\theta}_n - \theta_\ast \|^2  \geqslant \frac{16R^2} {\mu^2 n}\big( {
 \frac{\mu \sqrt{n}}{4R^2} +  \triangle } \big)^2
\bigg]^{1/2}   \bigg(  \|\theta_0 - \theta_\ast\|^2 +  9 \gamma^2 n R^2\bigg)
\mbox{ using Proposition~\ref{prop:boundp}.}
\EEAS
Moreover, if we denote by $t_0$ the largest solution of ${10}{\sqrt{t_0}} + 20 \square t_0    =         \frac{\mu \sqrt{n}}{4R^2}  $, we have:
\BEAS
\sqrt{t_0} & = & \frac{- 10 + \sqrt{ 100 + 20  \square  \frac{\mu \sqrt{n}}{ R}}}{40 \square}
 =  \frac{- 10 + 1 0 \sqrt{ 1 + 20  \square \frac{\mu \sqrt{n}}{100 R}}}{40 \square} \\
 & \geqslant & \frac{9}{40 \square}  \sqrt{ 20  \square \frac{\mu \sqrt{n}}{100 R}},
 \EEAS
as soon as $20  \square \frac{\mu \sqrt{n}}{100 R} \geqslant 100$, 
since if $q\geqslant 100$, $-1 + \sqrt{1+q} \leqslant \frac{9}{10} \sqrt{q}$. This implies that
\BEAS
& & \int_{\frac{16R^2}{ {\mu^2 n}}\big( {
 \frac{\mu \sqrt{n}}{4R^2} +  \triangle } \big)^2}^\infty 
\P \big[ 
\| \bar{\theta}_n - \theta_\ast \|^2  \geqslant u
\big] du
\\
& \leqslant & 
  \bigg[ 4 \exp(-t_0)
\bigg]^{1/2}   \bigg(  \|\theta_0 - \theta_\ast\|^2 +  9 \gamma^2 n R^2\bigg)
\\
& \leqslant & \frac{9}{2 t_0^2}  \bigg(  \|\theta_0 - \theta_\ast\|^2 +  9 \gamma^2 n R^2\bigg)
\mbox{ using } \exp(-\alpha) \leqslant \frac{9}{ 16 \alpha^2} \mbox{ for all } \alpha>0,\\
& \leqslant & \frac{9}{2} \frac{ 40^4 \square^4 100^2 R^4}{9^4  20^2 \square^2 \mu^2 n }
  \bigg[ \frac{9}{4} \square^2 / R^2 + \frac{\gamma \sqrt{n} }{3} \triangle   \bigg]\\
  & \leqslant & 
   686 \times 64 \frac{ \square^2R^2}{ \mu^2 n }
  \bigg[ \frac{9}{4} \square^2 + \frac{1}{6} \square \triangle   \bigg].
\EEAS
The second term in \eq{proofBB} is bounded exactly like in Appendix~\ref{app:boundF}, leading to:
 \BEAS
& & 
\int_{\triangle^2 \frac{16R^2}{ {\mu^2 n}} }^{\frac{16R^2}{ {\mu^2 n}}\big( {
 \frac{\mu \sqrt{n}}{4R^2} +  \triangle } \big)^2}
\P \big[ 
\| \bar{\theta}_n - \theta_\ast \|^2  \geqslant u
\big] du\\
& \leqslant &
 \int_{0}^\infty 4e^{-t} d \bigg(
 \frac{16R^2}{\mu^2 n}
\bigg[
 {10}{\sqrt{t}} + 20 \square t
 +  \triangle \bigg]^2 
\bigg)  \\
 & \leqslant &   
  \frac{64 R^2}{\mu^2 n}  \int_{0}^\infty e^{-t}   \bigg(
 100 + 400 \square^2 2t   + 400  \square \frac{3}{2} t^{1/2} + 20 \triangle \frac{1}{2} t^{-1/2} + 40   \triangle \square \bigg) dt
 \\
   & \leqslant &  
  \frac{64 R^2}{\mu^2 n}    \bigg(
 100 \Gamma(1) + 400 \square^2 2  \Gamma(2)    + 400  \square \frac{3}{2}  \Gamma(3/2)  + 20 \triangle \frac{1}{2}\Gamma(1/2) + 40   \triangle \square \Gamma(1) \bigg) dt
  \\
& \leqslant &   
  \frac{64 R^2}{\mu^2 n}  \bigg(
 100   + 400 \square^2 2       + 400  \square \frac{3}{2}  \frac{1}{2} \sqrt{\pi} + 20 \triangle \frac{1}{2}\sqrt{\pi} + 40   \triangle \square \bigg)  .
  \EEAS
  We can now put all elements together to obtain, from \eq{proofBB}:
   \BEAS
& & \E \| \bar{\theta}_n - \theta_\ast \|^2 \\
& \leqslant &  \frac{64 R^2}{\mu^2 n}  \bigg(
 100   + 400 \square^2 2       + 400  \square \frac{3}{2}  \frac{1}{2} \sqrt{\pi} + 20 \triangle \frac{1}{2}\sqrt{\pi} + 40   \triangle \square \bigg)  \\
 & & \hspace*{6cm} 
 + \frac{16R^2}{ {\mu^2 n}} \triangle  ^2 +  686 \times 64 \frac{ \square^2R^2}{ \mu^2 n }
  \bigg[ \frac{9}{4} \square^2 + \frac{1}{6} \square \triangle   \bigg] \\
  & \leqslant & 
    \frac{64 R^2}{n\mu^2}
\bigg[
\frac{1}{4}\triangle^2 + 100 + 800 \square^2 + 532 \square + 32  {\triangle}  + 40  {\triangle} \square + 686 \frac{9}{4} \square^4 + 686 \frac{\triangle \square^3}{6}
\bigg].
\EEAS
For $\gamma = \frac{1}{2R^2 \sqrt{N}}$, with $\alpha = R \| \theta_0 - \theta_\ast\|$, $\square = 1$ and  $\triangle = 6 \alpha^2 + 6 \alpha$, we get
\BEAS
\E \|\bar{\theta}_N - \theta_\ast\|^2 
&\!\!\!\! \leqslant \!\!\!\!& 
 \frac{  8R^2}{N\mu^2}
 \bigg[
 2\triangle^2 + 8 \triangle ( 32+ 40 +115) + 8 ( 100+800+532+1544)
 \bigg]
\\
& \!\!\!\!\leqslant \!\!\!\!& 
 \frac{  8R^2}{N\mu^2}
 \bigg[
 2\triangle^2 + 1496 \triangle + 23808
 \bigg] 
\\
& \!\!\!\!\leqslant \!\!\!\!& 
 \frac{  8R^2}{N\mu^2}
 \bigg[
72 \alpha^4 + 144 \alpha^3 + 72 \alpha^2 + 1496 \times 6 \alpha^2 + 1496 \times 6 \alpha   + 23808
 \bigg] \\
  & \!\!\!\!\leqslant \!\!\!\!& 
  \frac{  R^2}{N\mu^2}
 \bigg[
1296  \alpha^4 + 18144 \alpha^3 + 95256 \alpha^2 + 222264 \alpha   + 194481
 \bigg] 
\! = \!
 \frac{  R^2}{N\mu^2}
 \big( 6 \alpha + 21 \big)^4.
\EEAS

The previous bound is valid as long as $ \frac{\mu \sqrt{N}}{R} \geqslant \frac{10000}{20}=500$. If it is not satisfied, then Lemma~\ref{prop:old} shows that it is still valid.

\vskip 0.2in
\bibliography{bach14a}

       \end{document}